\documentclass[12pt,oneside]{amsart}
\usepackage[inner=25mm,outer=25mm,tmargin=27mm,bmargin=27mm]{geometry}
\usepackage[utf8]{inputenc}
\usepackage[T1]{fontenc}
\usepackage{amssymb}
\usepackage{geometry}
\usepackage{amsmath}
\usepackage{amsthm}
\usepackage{graphicx}
\usepackage{color,xcolor}
\usepackage[normalem]{ulem}
\usepackage[english]{babel}
\usepackage[onehalfspacing]{setspace}
\usepackage{verbatim}
\usepackage[hidelinks]{hyperref}
\usepackage{mathtools}

\newcommand{\IR}{\ensuremath{\mathbb{R}}}
\newcommand{\IN}{\ensuremath{\mathbb{N}}}

\newcommand{\IP}{\ensuremath{\mathbb{P}}}
\newcommand{\IE}{\ensuremath{\mathbb{E}}}

\newcommand{\IS}{\ensuremath{\mathbb{S}}}

\newcommand{\dint}{{\rm d}}

\newcommand{\norm}[1]{\left\Vert#1\right\Vert}

\newcommand{\scalar}[2]{\left\langle#1,#2\right\rangle}

\renewcommand{\rho}{\varrho}

\newcommand{\set}[1]{\left\{#1\right\}}
\newcommand{\abs}[1]{\left|#1\right|}
\newcommand{\brackets}[1]{\left(#1\right)}
\renewcommand{\d}{{\rm d}}

\renewcommand{\phi}{{\varphi}}

\newcommand{\dk}[1]{#1} 
\newcommand{\diag}{\mathop{\mathrm{diag}}}

\DeclareMathOperator\rad{rad}

\newtheorem{thm}{Theorem}
\newtheorem{cor}[thm]{Corollary}
\theoremstyle{plain}
\newtheorem{lemma}[thm]{Lemma}

\newtheorem{prop}[thm]{Proposition}
\theoremstyle{definition}
\newtheorem{rem}[thm]{Remark}
\newtheorem*{ack}{Acknowledgements}

\title[Random sections of ellipsoids and random information]{Random sections of ellipsoids and 
the power of\\
random information} 

\author{Aicke Hinrichs 
\and David Krieg 
\and Erich Novak 
\and Joscha Prochno 
\and Mario Ullrich}

\address[A.~Hinrichs, D. Krieg, M.~Ullrich]{Institut f\"ur Analysis,
Johannes Kepler Universit\"at Linz,
Altenbergerstrasse 69, 4040 Linz, Austria}

\address[E.~Novak]{\dk{Institut f\"ur Mathematik, 
Friedrich Schiller Universit\"at Jena, 
Ernst-Abbe-Platz 2, 07743 Jena, Germany}}

\address[J.~Prochno]{Institut f\"ur Mathematik \& Wissenschaftliches Rechnen, 
Karl-Franzens-Universit\"at Graz, 
Heinrichstrasse 36, 8010 Graz, Austria}


\date{\today}

\begin{document}

\begin{abstract}
We study the circumradius of the intersection of an $m$-dimensional ellipsoid~$\mathcal E$
with semi-axes $\sigma_1\geq\dots\geq \sigma_m$ with random subspaces of
%\dk{almost full dimension}. % 
codimension $n$, \dk{where $n$ can be much smaller than $m$}.
We find that, under certain assumptions on %\dk{the decay of} 
$\sigma$, 
this random radius 
$\mathcal{R}_n=\mathcal{R}_n(\sigma)$ %of the random section} 
is of the same order as the minimal such 
radius $\sigma_{n+1}$ with high probability. In other situations $\mathcal{R}_n$ 
is close to the maximum~$\sigma_1$. The random variable $\mathcal{R}_n$ naturally corresponds to the 
worst-case error of the best algorithm based on random information for 
$L_2$-approximation of functions from a compactly embedded 
Hilbert space $H$ with unit ball $\mathcal E$. 
In particular, $\sigma_k$ is the $k$th largest singular value of the embedding 
$H\hookrightarrow L_2$. In this formulation, one can also consider the 
case $m=\infty$ and we prove that 
random information behaves very differently depending 
on whether $\sigma \in \ell_2$ or not. 
For $\sigma \notin \ell_2$ \dk{we get
	$\IE[\mathcal{R}_n] = \sigma_1$ and} random information is completely useless. 
For $\sigma \in \ell_2$ the expected \dk{radius} %of random information 
tends to zero at least 
	at rate $o(1/\sqrt{n})$  as $n\to\infty$. 
	In the important case
 \[
  \sigma_k \asymp k^{-\alpha} \ln^{-\beta}(k+1),
 \]
 where $\alpha > 0$ and $\beta\in\IR$ 
 \dk{(which corresponds to various Sobolev embeddings)},
we prove 
$$
\mathbb E [\mathcal{R}_n(\sigma)] \asymp
  \left\{\begin{array}{cl}
  	\sigma_1 	
        &
        \text{if} \quad \alpha<1/2 \text{\, or \,} \beta\leq\alpha=1/2,
        \vspace*{2mm}
        \\
	  \sigma_{n+1} \, \sqrt{\ln(n+1)}  \quad
        &
        \text{if} \quad \beta>\alpha=1/2,
        \vspace*{2mm}
        \\
       \sigma_{n+1}  
        &
        \text{if} \quad \alpha>1/2.
        \end{array}\right.
$$
In the proofs we use
a comparison result for Gaussian 
processes \`a la Gordon, exponential estimates for sums of chi-squared random variables, 
and estimates for the extreme singular values of (structured) Gaussian random matrices.
\dk{The upper bound is constructive. It is proven for the
worst case error of a least squares estimator.}
\end{abstract}

\maketitle
%   \tableofcontents

\newpage

% % % % % % % % % % % % % % % % % % % %
\section{Introduction}
% % % % % % % % % % % % % % % % % % % %

%Problem AGA

\noindent
We are interested in the circumradius of the intersection of a centered
ellipsoid  $\mathcal E$ in $\IR^m$  with a random subspace $E_n$ of codimension $n$, 
where $n$ can be much smaller than $m$.
While the maximal radius is the length of the largest
semi-axis $\sigma_1$, the minimal radius is the length 
of the $(n+1)$-st largest semi-axis $\sigma_{n+1}$.
But how large is the radius of a typical intersection?
Is it comparable to the minimal or the
maximal radius or does it behave 
completely different?
We prove that the radius of a random intersection satisfies 
$$
   \rad ( \mathcal E \cap E_n) \leq  \frac{c}{\sqrt{n}} 
   \bigg(\sum_{j\geq  n/4 }\sigma_j^2\bigg)^{1/2}
$$
with overwhelming probability, where $c\in(0,\infty)$ is an absolute constant.
For many sequences $\sigma$ of semi-axes the right-hand side is of the same order as $\sigma_{n+1}$.
This means that a typical intersection has radius comparable to the smallest one.
One example are semi-axes  of length $\sigma_j = j^{-\alpha}$ of polynomial decay  $\alpha>1/2$.

If the sequence $\sigma$ decays too slowly, this is no longer true and 
we find that a typical intersection often has radius comparable to the largest one.
Indeed, if the ellipsoid is `fat' in the sense that the semi-axes satisfy $\| \sigma \|_2 \ge c \sqrt{n} \sigma_1$, 
then we show that 
$$
   \rad ( \mathcal E \cap E_n) \geq \sigma_1 / 2
$$
with overwhelming probability, where $c\in(0,\infty)$ is an absolute constant.
An example are semi-axes  of length $\sigma_j = j^{-\alpha}$ of polynomial decay  $\alpha \le 1/2$.
Altogether, we obtain
 \[
\IE [ \rad ( \mathcal E \cap E_n)] \asymp
  \left\{\begin{array}{cl}
  	\sigma_1 	
        &
        \text{if} \quad \alpha \leq 1/2,
        \vspace*{2mm}
        \\
       \sigma_{n+1}  
        &
        \text{if} \quad \alpha>1/2,
        \end{array}\right.
 \]
where $\asymp$ denotes equivalence up to positive constants not depending on $n$ and $m$.

The study of diameters of sections of symmetric convex bodies 
with a lower-dimensional subspace has been initiated  
by Giannopoulos and Milman \cite{GM1998,GM1997} and further advanced 
in the subsequent works of Litvak and Tomczak-Jaegermann \cite{LT2000}, Giannopoulos, 
Milman, and Tsolomitis \cite{GMT2005}, or Litvak, Pajor, and Tomczak-Jaegermann \cite{LPT06}. 
However, as has already been pointed out in \cite{GM1998,GM1997}, one cannot expect these bounds 
to be sharp for the whole class of symmetric convex bodies as is indicated by ellipsoids 
with highly incomparable semi-axes for which the diameter of sections of proportional 
dimension does not concentrate around some value \cite[Example 2.2]{GM1997}. 
Moreover, the focus in these papers was on subspaces of proportional codimension, 
whereas we are \dk{mainly} interested in subspaces with small codimension 
such as $m=n^2$ or $m= 2^n$ \dk{or even $m=\infty$}.

%%%%%%
\medskip

Our %initial 
motivation %leading us to this geometric problem 
%has been completely different.
%Its origin lies in
\dk{comes from} the theory of in\-for\-ma\-tion-based complexity (IBC).
%where one often wants 
\dk{In IBC we often want} to approximate the solution of a linear problem based on $n$ pieces of information about the unknown problem instance.
We refer to \cite{NW08,NW10,NW12} for a detailed exposition.
\dk{It is usually assumed} that some kind of oracle is available which grants us this information at our request. We call this oracle $n$ times to get $n$ \dk{well-chosen} pieces of information,
% and try to choose clever questions 
%such that the answers of the oracle are most meaningful in order 
\dk{trying} to obtain \emph{optimal information} about the problem instance. 
Often, however, this model \dk{is too idealistic}. % does not match reality. 
There might be no such oracle at our disposal and
the information comes in randomly. 
We simply have to work with the information at hand.
This is in fact a standard assumption in learning theory and uncertainty 
quantification, see \cite{SC08}. 
It may also happen that an oracle is available but we simply do not know what to ask in order to obtain optimal information. In such a case, 
%it immediately suggests itself 
it seems natural to ask random questions. %What we obtain is 
%\dk{In both cases, we obtain \emph{random information}
%instead of optimal information.}
Both scenarios suggest the analysis of \emph{random information} 
and the question how it compares to optimal information.
For a survey of some classical results as well as new results
see~\cite{HKNPU19-survey}. Here we study the case of $L_2$-approximation
of vectors or functions from a Hilbert space. 
%For a survey on some classical and fundamental results of this 
%type, we refer to~\cite{HKNPU19-survey}. 
%Here, we initiate a systematic study of such problems, 
%starting with the special case of $\ell_2$-approximation 
%of vectors/functions from a Hilbert space.

% The ellipsoid intersection problem originates in the following approximation problem:
More precisely, we \dk{consider the problem of recovering} %want to recover
$x\in \mathcal E$ 
from the data $N_n( x )\in \IR^n$ 
\dk{which is obtained from an} information mapping $N_n \in \IR^{n \times m}$
and measure the error in the Euclidean norm.
The power of the information mapping is \dk{given} % measured
by its radius, which is the worst case error of the best
recovery algorithm based on $N_n$,
that is,
$$
  \rad(N_{n},\mathcal E)
 = \adjustlimits\inf_{\phi\colon\IR^n\to \IR^m} 
 \sup_{x\in \mathcal E} \norm{\phi(N_{n}(x))-x}_2.
$$
%It is known that, for linear problems in Hilbert spaces, 
%D: Was soll ein lineares Problem in H Raeumen sein?
\dk{For problems of this type, it is known that}
the worst data is the zero data, resulting in
$$
  \rad(N_{n},\mathcal E)
 = 
 \sup_{x \in \mathcal E \cap E_n} \norm{x}_2,
$$
where $E_n$ is the kernel of $N_n$, see \cite{CW04,NW08,TWW88}. 
Thus, \dk{if $N_n$ is a standard Gaussian matrix,} 
we indeed arrive at the same problem as above.
The radius of a random intersection is the worst case error of the best algorithm based
on \dk{Gaussian} random information,
whereas the radius of the minimal intersection is the worst case error of the best algorithm based
on optimal information.
So the geometric questions above translate as follows:
How good is random information?
Is it comparable to the optimal information or is it much worse?
%
%The geometric statements from above also hold for the radius $\rad ( N_{n},\mathcal E)$
%of random information instead of $ \rad ( \mathcal E \cap E_n)$. 
\dk{The answers are the same.}
For instance, for polynomial decay $\sigma_j = j^{-\alpha}$, we have
 \[
\IE [ \rad ( N_{n},\mathcal E)] \asymp
  \left\{\begin{array}{cl}
  	\sigma_1 	
        &
        \text{if} \quad \alpha \leq 1/2,
        \vspace*{2mm}
        \\
       \sigma_{n+1}  
        &
        \text{if} \quad \alpha>1/2.
        \end{array}\right.
 \]

As a matter of fact, \dk{the results for the radius of random information} 
even hold when $m=\infty$, where our geometric interpretation fails.
\dk{Namely, for any} $\sigma \notin \ell_2$,
we obtain that %random information is completely useless, i.e., 
$\IE [ \rad ( N_{n},\mathcal E)] = \sigma_1$
\dk{and random information is completely useless.}
For $\sigma \in \ell_2$
the expected radius of random information tends to zero 
%at least 
%at rate $o(1/\sqrt{n})$ as $n\to\infty$. 
\dk{with the same polynomial rate as the radius of optimal information.
The proof of this upper bound is constructive.
We present a least squares estimator
based on random information that is almost as good as 
the optimal algorithm based on optimal information.}
%we obtain that random information is completely useless if $\sigma\not\in\ell_2$.}

\begin{rem} 
	\dk{Using isomorphisms, our results can easily be transferred to any compact embedding $S$
	of a Hilbert space $H$ into a separable $L_2$-space.
	That is, we may also consider the problem of approximation an unknown function $f$
	from the unit ball $\mathcal E$ of $H$ in the $L_2$-norm.
	In this case, optimal information is given by the generalized Fourier coefficients
	\[
	 N_n^*(f)=\big(\langle f, b_i\rangle_2\big)_{i=1\hdots n}
	\]
	where $b_i$ is the $L_2$-normalized eigenfunction belonging to the $i$th largest eigenvalue of the operator $S^*S$.
	The radius of optimal information $\sigma_{n+1}$ is the square-root of the $(n+1)$st largest eigenvalue.
	Random information, on the other hand, is given by
		\[
	N_n(f)=\Big(\sum_{j=1}^\infty g_{ij} \langle f, b_j\rangle_2\Big)_{i=1\hdots n},
	\]
	where the $g_{ij}$ are independent standard Gaussian variables.
	Equivalently, if $\sigma\in\ell_2$, we have
	\[
	N_n(f)=\big(\langle f, h_i\rangle_H\big)_{i=1\hdots n},
	\]
	where the $h_i$ are iid Gaussian fields on $H$ whose
	correlation operator $C\colon H\to H$ is defined by $C b_j=\sigma_j^2 b_j$.
%	Instead of $\ell_2$ we may also take a separable $L_2$ space since both spaces 
%	are isometrically isomorphic. 
%Then we may study a compact  embedding 
%$H\hookrightarrow L_2$ of a Hilbert space $H$ and denote the unit ball of $H$ by
%$\mathcal E$. 
 The results are the same.
 In particular, random information is (almost) as good as 
optimal information
 as long as $\sigma\in\ell_2$.
 An important case, which is} often needed in approximation theory 
and complexity studies, 
	are Sobolev embeddings, i.e., $H$ is a Sobolev space of functions 
	that are defined 
	on a bounded domain in $\IR^d$. 
	It is well known that then the singular values behave as 
$ 
  \sigma_k \asymp k^{-\alpha} \ln^{-\beta}(k+1) 
$,  
where $\alpha$ and $\beta$ depend on the smoothness and the dimension $d$
and the condition $\sigma\in\ell_2$ means that the functions in $H$ are continuous,
see also~\cite{CS1990}.
\end{rem}

\begin{rem}
The phenomenon, that the results very much depend on whether $\sigma$ is 
square summable or not, is known from a related problem
that was studied earlier in several papers. 
There $\mathcal{E}$ is the unit ball of a reproducing kernel Hilbert space $H$.
\dk{That is}, $H\subseteq L_2(D)$ consists of functions 
on a common domain $D$ and
function evaluation $f \mapsto f(x)$ is 
a continuous functional on $H$ for every $x\in D$. 
Again, the optimal linear information 
$N_n$ for the $L_2$-approximation problem is given by the 
%singular value decomposition 
\dk{generalized Fourier coefficients}
and has radius $\sigma_{n+1}$. 
This information might be \dk{hard to get}
%difficult to implement 
and hence one might allow only \dk{function evaluations, i.e.,} information of the form 
\[
N_n (f) = \big( f(x_1), \dots , f(x_n)\big)\,, \qquad x_i\in D.
\]
%The goal is to relate 
%the power of standard information~$\lstd$ 
%to the power of the class~$\lall$ of all continuous 
%linear functionals. 
The goal is to relate the power of 
function evaluations
to the power of all continuous 
linear functionals.
%Ideally one would like to
%prove that the power of $\lstd$ is roughly the same as the power of $\lall$.
Ideally one would like to
prove that their power is roughly the same.
Unfortunately, in general
this is \emph{not} true.
In the case $\sigma \notin \ell_2$
the convergence of optimal algorithms 
that may only use function values
can be arbitrarily slow 
\cite{HNV08}.
The situation is much better if we assume that $\sigma \in \ell_2$.
It was \dk{already} shown in 
\cite{WW01} 
and
\cite{KWW09}
that function values \dk{cannot be much worse than} general 
linear information \dk{in this case}.
%\dk{see also \cite[Chapter 26]{NW12}}.
We refer to \cite[Chapter 26]{NW12} for a presentation of these results. 
\dk{Indeed, based on the proof technique of the present paper,
%these results have recently been improved in \cite{KU19}.
%There, it is shown 
it was recently proven in \cite{KU19} 
that the polynomial order of convergence for function values and general
linear information is the same for all $\sigma\in\ell_2$.}
\end{rem} 

The rest of the paper is organized as follows. 
In Section~\ref{sec:problem and results}
we discuss the relation between the geometric problem
and the IBC problem in more detail. 
%We give three equivalent versions of the problem.
%The version involving Gaussian information 
%easily generalizes to the infinite dimensional setting. 
We give general upper bounds (Theorem \ref{thm:upper bound random section A} \dk{and \ref{thm:secondUB}})
and lower bounds (Theorem \ref{thm:lower B})
for the radius of random information in terms of the sequence $\sigma$
which hold with high probability.
We derive the $\ell_2$-dichotomy discussed above (Corollary \ref{cor:l2 not l2}) and
apply the general theorems to sequences of polynomial
decay (Corollary \ref{cor:polynomial}) and %of 
exponential decay (Corollary \ref{cor:exponential}).
The proofs %of these results 
are contained in Section \ref{sec:proofs}.
%, which is
%partitioned into the subsections containing the proof of the upper bound,
%the proof of the lower bound, and the proof of the corollaries.
%%For the proofs we use
%%a comparison result for Gaussian 
%%processes \`a la Gordon, exponential estimates for sums of chi-squared random variables, 
%%or estimates for the extreme singular values of (structured) Gaussian random matrices.
We add a final section about alternative approaches.
%In particular, 
We show an upper bound via the \dk{lower $M^\ast$-estimate}
and an elementary lower bound. 
%Although these bounds are not as sharp and general 
%as those presented before, they highlight other aspects of the problem.
\dk{These bounds are slightly weaker, 
but give a better insight into the geometric aspect of the problem.}

% % % % % % % % % % % % % % % % % % % %
\section{Problem and results}
\label{sec:problem and results}
% % % % % % % % % % % % % % % % % % % %

\noindent
\dk{We consider the ellipsoid 
$$
 \mathcal E_\sigma = \left\{ x \in \ell_2 \colon 
 %\vert x_j \vert \le \sigma_j \text{ for all } j\in\IN \text{ and } 
 \sum_{j\in\IN} \left(\frac{x_j}{\sigma_j}\right)^2 \leq 1 \right\}
$$
with semi-axes of lengths $\sigma_1 \ge \sigma_2 \ge \dots \ge 0$.\footnote{For convenience, 
we use the convention that $a/0=\infty$ for all $a>0$ and $0/0=0$.}
This is the unit ball of a Hilbert space, which we denote by $H_\sigma$.
We study the problem of recovering an unknown vector
$ x\in \mathcal E_\sigma$ from $n$ pieces of information,
where we want to guarantee a small error in $\ell_2$.
The information about $ x\in \mathcal E_\sigma$
is given by the outcome $L_1(x),\hdots,L_n(x)$ 
of $n$ linear functionals.\footnote{Note that we also 
consider unbounded functionals,
meaning that $L_i(x)\in\IR\cup\{{\rm NaN}\}.$}
The mapping $N_n=(L_1,\hdots,L_n)$ is called the information mapping.
%$N_n:H_\sigma\to \IR^n$.
A recovery algorithm is a mapping $A_n:H_\sigma\to\ell_2$ of the form $A_n=\phi\circ N_n$, where 
$\phi$ maps $N_n(x)$ back to $\ell_2$.
%\footnote{If $N_n$ is unbounded 
%and $N_n(x)$ is not defined for some $x\in H_\sigma$, we simply set $A_n(x)=0$. 
%This is clearly the best we can do 
%since we cannot distinguish $x$ from $-x$ for $N_n(-x)$ is also undefined.} 
The worst case error of the algorithm is given by
\[
 e(A_n)=\sup_{x\in \mathcal E_\sigma} \norm{A_n(x)-x}_2.
\]
The quality of the information mapping is measured by its radius, 
which is the worst case error of the best
recovery algorithm based on the information $N_n$,
i.e.,
\[
 \rad(N_n,\mathcal E_\sigma) :=\inf_{\phi\colon {\rm im}(N_n)\to \ell_2} e(\phi\circ N_n).
\]
Note that this is a linear problem over Hilbert spaces as described in \cite[Section~4.2.3]{NW08}.
In particular, we have the relation
\[
\rad(N_n,\mathcal E_\sigma)
\,\ge\, \sup\left\{ \norm{x}_2 \colon x \in \mathcal E_\sigma \text{ with } N_n(x)=0 \right\}
\]
with equality for all bounded information mappings $N_n$.
We refer to \cite{CW04,NW08,TWW88}. 
It is easy to see that \emph{optimal information} is given by the mapping
\[
 N_n^*:H_\sigma \to \IR^n, \quad N_n^*(x)=(x_1,\hdots,x_n),
\] 
which satisfies
\[
 \rad(N_n^*,\mathcal E_\sigma)
 = \inf_{N_n \text{ linear}} \rad(N_n,\mathcal E_\sigma) = \sigma_{n+1}.
\]
We want to compare this to the radius of \emph{random information},
which is given by a random matrix $G_n\in\IR^{n\times\infty}$
with independent standard Gaussian entries.
That is, we study the random variable
\[
 \mathcal R_n(\sigma) = \rad(G_n,\mathcal E_\sigma).
\]
Clearly, we always have $\mathcal R_n(\sigma) \in [\sigma_{n+1},\sigma_1]$.
There are two alternative interpretations of the quantity $\mathcal R_n(\sigma)$
if the sequence $\sigma$ is finite in the sense that $\sigma_j=0$ for all $j>m$.}

\dk{%
\medskip
% % % % % % % % % % % % % % % %
\noindent\textbf{Variant 1.}
% % % % % % % % % % % % % % % %
The quantity $\mathcal R_n(\sigma)$ is the circumradius of the $(m-n)$-dimensional ellipsoid
that is obtained by slicing the $m$-dimensional ellipsoid $\mathcal E_\sigma$ with a
subspace $E_n$ that is uniformly distributed on the
Grassmannian manifold $\mathbb{G}_{m,m-n}$ of $n$-codimen\-sional subspaces in $\IR^m$
(equipped with the Haar probability measure).
That is,
$$ 
  \mathcal R_n(\sigma) =
  \rad\brackets{\mathcal E_\sigma\cap E_n} 
  = \sup \big\{ \|x\|_2 \colon x \in \mathcal E_\sigma\cap E_n \big\}.
$$
%where $\|x\|_2$ denotes the standard Euclidean norm of $x\in \IR^m$.
%This is the radius of the smallest Euclidean
%ball that contains the intersection ellipsoid,
%or equivalently the length of its largest semi-axis.
This easily follows from the fact that the kernel of the matrix $G_n$ 
(when restricted to $\IR^m$) is uniformly distributed on the Grassmannian.
%Observe that $ \rad\brackets{\mathcal E_\sigma^m\cap E_n}\in [\sigma_{n+1},\sigma_1]$.
%Clearly, the alignment of the ellipsoid with the standard axes of $\IR^m$
%and the order of the semi-axes is no essential assumption.

\medskip
% % % % % % % % % % % % % %
\noindent\textbf{Variant 2.}
% % % % % % % % % % % % % %
Since the radius of information is invariant under a (component-wise) scaling of the information mapping,
we also have
\[
 \mathcal R_n(\sigma) = \rad(U_n,\mathcal E_\sigma),
\]
where $U_n$ is obtained from $G_n$ by erasing all but the first $m$ columns
and then normalizing the rows in $\ell_2$.
That is, in the finite-dimensional case, we may just as well study
the quality of the information given by $n$ coordinates
in random directions that are independent and uniformly distributed on the sphere $\mathbb S^{m-1}$.%
}

\medskip

\dk{We want to give upper and lower bounds for $\mathcal R_n(\sigma)$ which hold with high probability.
Clearly, upper bounds are stronger if they are proven for the error $e(A_n)$ 
of a concrete algorithm $A_n=\phi\circ G_n$.
Here, we consider a least squares estimator.
%We define $V_k\subset \ell_2$ as the subspace of all
%$x\in\ell_2$ such that $x_j=0$ for all $j> k$,
We define $\psi$ as the restriction of $G_n$ to $\IR^k$,
where $k\le n$ is of order $n$ and will be specified later.
%(The precise value depends on the setting). 
Note that we identify $\IR^k$ with the space of all
$x\in\ell_2$ such that $x_j=0$ for all $j> k$.
We then take $\phi=\psi^{+}$ 
%consider the restriction 
%$\psi$ of $G_n$ to $V_k$, 
%which may clearly be represented as a matrix 
%$\psi\in\IR^{n\times k}$ by identifying $V_k$ with $\IR^k$,  
%and take 
%\[
% \phi:=\psi^{+}\colon\IR^n\to \IR^k \simeq V_k,
%\] 
where $\psi^+$ is the Moore-Penrose inverse of $\psi\in\IR^{n\times k}$.
%and consider the algorithm $A_n=\phi\circ G_n$. 
This algorithm satisfies, almost surely, that $A_n(x)=x$ for all 
$x\in \IR^k$. Moreover, one may write
\begin{equation}\label{eq:An}
A_n(x) \,=\,  \underset{y\in \IR^k}{\operatorname{arg\, min}}\; \left\|G_n(x-y)\right\|_2, 
\qquad x\in\IR^m.
\end{equation}

Let us now present the results.
We note that it is not an essential assumption that the vector of semi-axes is non-increasing
or even that the semi-axes are aligned with the standard basis of the Euclidean space.
It simply eases the notation.}

\begin{thm}
\label{thm:upper bound random section A}
Let $\sigma\in\ell_2$ be non-increasing \dk{and let $n\in\IN$.
Then the following estimate holds with probability at least $1-2\exp(-n/100)$,}
  $$
   \dk{\mathcal R_n(\sigma) 
   \,\le\, e(A_n)
   \,\le\, \frac{221}{\sqrt{n}} 
   \bigg(\sum_{j\geq\lfloor n/4\rfloor}\sigma_j^2\bigg)^{1/2}.}
  $$
\end{thm}
  
\dk{This estimate turns out to be useful for sequences $\sigma$ of polynomial decay.
For sequences of exponential decay, we add a second upper bound.
It is better suited for such sequences
since the starting index $\lfloor n/4\rfloor$ of the sum in the upper bound is replaced by $n$.}

\dk{\begin{thm}
\label{thm:secondUB}
Let $\sigma\in\ell_2$ be non-increasing.
Then, for all $n\in\IN$ and
$c,s\in[1,\infty)$ we have
  $$
   \IP\left[
   \mathcal R_n(\sigma) \le e(A_n) \le 14sn
   \bigg(\sum_{j>n}\sigma_j^2\bigg)^{1/2}\,
   \right]
   \ge 1- e^{-c^2 n} - \frac{c \sqrt{2e}}{s}.
$$
\end{thm}}

\dk{On the other hand, we obtain the following lower bound for $\mathcal R_n(\sigma)$.
This lower bound is even satisfied for the smaller quantity
\[
 \mathcal R_n^{(1)}(\sigma)
 := \sup\left\{ x_1 \colon x\in \mathcal E_\sigma,\, G_n(x)=0\right\},
\]
which corresponds to the difficulty of the easier problem of recovering just the first coordinate
of $x\in\mathcal{E}_\sigma$ from Gaussian information.}

\begin{thm}
\label{thm:lower B}
Let $\sigma\in\ell_2$ be non-increasing, $\varepsilon\in(0,1)$ and $n,k\in\IN$ \dk{with} 
%$\sigma_k\neq 0$ and
\[
\sum_{j>k} \sigma_j^2 \geq \frac{3n\sigma_k^2}{\varepsilon^2}.
\]
%Then, for all $0<\delta<1$, 
%we have
%\[
%\IP\left[ \mathcal R_n^{(k)}(\sigma) \,\le\, 
%\sigma_k \left(1\,-\,\sqrt{\frac{(1+\delta)n}{(1-\delta)C_k}}\,\right)\right]
%\leq 
%5 \exp\brackets{-(\delta/4)^2\, \min\set{n, C_k}}.
%\]
%In particular, 
%For $k\in\IN$ satisfying 
%$C_k\geq 3n\varepsilon^{-2}$ for some $0<\varepsilon<1$, 
Then \dk{it holds with probability at least $1- 5\exp\brackets{-n/64}$ that}
\[
 \dk{\mathcal R_n(\sigma) \,\ge\, \mathcal R_n^{(1)}(\sigma) \,\ge\, \sigma_k(1-\varepsilon).}
\]
\end{thm}

%\dk{It} will become apparent in the proof that
%the lower bound of Theorem~\ref{thm:lower B}
%already holds for the easier problem of recovering just the $k$-th coordinate
%of $x\in\mathcal{E}_\sigma$.
As a consequence of \dk{these} theorems, 
we obtain that random information is useful if and only if $\sigma$
\dk{is square summable}.

\begin{cor}
 \label{cor:l2 not l2}
 If $\sigma\not\in\ell_2$, then $\mathcal R_n(\sigma)=\sigma_1$
  holds almost surely for all $n\in\IN$.
 \dk{On the other hand,} if $\sigma\in\ell_2$, then
 $$
  \lim_{n\to\infty}\sqrt{n}\,\IE[ \mathcal R_n(\sigma)] = 0.
 $$
\end{cor}

Before we present the proofs of our main results, 
let us provide some of the results on the expected radius that
follow from our main results for special sequences. 
\dk{For sequences $(a_n)$ and $(b_n)$ we write $a_n\preccurlyeq b_n$ if there is a constant $C>0$ such that $a_n \leq C\, b_n$ for all $n\in\IN$. We write $a_n\asymp b_n$ in the case that both $a_n\preccurlyeq b_n$ and $b_n\preccurlyeq a_n$.}
%For sequences $(a_{n,m})$ and $(b_{n,m})$, we write $\preccurlyeq$ to indicate that there exists a constant $C\in(0,\infty)$ such that $a_{n,m} \leq C\, b_{n,m}$ for all $n,m$. We shall write $\asymp$ in the case that there are two constants $C_1,C_2\in(0,\infty)$ such that $C_1\, a_{n,m} \leq b_{n,m} \leq C_2\, a_{n,m}$ for all $n,m$. 
We start with the case of polynomial decay.

%\begin{cor}
% \label{cor:polynomial}
% Let $m\in\IN\cup \{\infty\}$.
% Assume that $\sigma$ is non-increasing and 
% \[
%  \sigma_n \asymp n^{-\alpha} \ln^{-\beta}(n+1)
% \]
% for some $\alpha\ge 0$ and $\beta\in\IR$ (with $\beta\ge0$ for $\alpha=0$). 
%Then
% %there is some $c\in(0,\infty)$, depending only on $\alpha$
% %and $\beta$, such that
%\medskip
% \[
%  \IE[ \mathcal R_{n,m}(\sigma)] \asymp
%  \left\{\begin{array}{cll}
%  		\sigma_1
%        &
%        \text{for} \quad n< c_m,
%        &
%        \text{if} \quad \alpha<1/2 \text{\, or \,} \beta\leq\alpha=1/2,
%        \vspace*{2mm}
%        \\
%        \sigma_{n+1}\sqrt{\ln(n+1)}
%        \quad
%        &
%        \text{for} \quad n<\sqrt{m},
%        &
%        \text{if} \quad \beta>\alpha=1/2,
%        \vspace*{2mm}
%        \\
%        \sigma_{n+1}
%        &
%        \text{for} \quad n<m,
%        &
%        \text{if} \quad \alpha>1/2,
%        \end{array}\right.
% \]
%with
%\medskip
% %where we can choose 
%\[
% c_m=
% \left\{\begin{array}{cl}
%        c\,m^{1-2\alpha} \ln^{-\max\{2\beta,0\}}m
%        &
%        \text{for} \quad \alpha<1/2,
%        \\
%        c \ln^{1-\max\{2\beta,0\}} m
%        &
%        \text{for} \quad \beta<\alpha=1/2,
%        \\
%        c \ln\ln m
%        &
%        \text{for} \quad \beta=\alpha=1/2,
%        \end{array}\right.
% \]
%where $c\in(0,\infty)$ is an absolute constant.
%\end{cor}

\dk{\begin{cor}
 \label{cor:polynomial}
 Let $\sigma$ be non-increasing with 
 $
  \sigma_n \asymp n^{-\alpha} \ln^{-\beta}(n+1)
 $
 for some $\alpha\ge 0$ and $\beta\in\IR$ (where $\beta\ge0$ for $\alpha=0$). 
Then
\medskip
 \[
  \IE[ \mathcal R_n(\sigma)] \asymp
  \left\{\begin{array}{cll}
  		\sigma_1
        &
        \text{if} \quad \alpha<1/2 \text{\, or \,} \beta\leq\alpha=1/2,
        \vspace*{2mm}
        \\
        \sigma_{n+1}\sqrt{\ln(n+1)}
        \quad
        &
        \text{if} \quad \beta>\alpha=1/2,
        \vspace*{2mm}
        \\
        \sigma_{n+1}
        &
        \text{if} \quad \alpha>1/2.
        \end{array}\right.
 \]
\end{cor}}

Similar results can be derived for the finite-dimensional case
(i.e., $\sigma_j=0$ for \mbox{$j>m$})
under the condition that $m$ is large enough in comparison
to $n$. Details can be found in the 
thesis~\cite[Corollaries~4.31--4.33]{Kr19PhD}.
This means that random information is 
just as good as optimal information
if the singular values decay with
a polynomial rate greater than $1/2$.
The size of a typical intersection ellipsoid
is comparable to the size of the smallest 
intersection.
On the other hand,
if the singular values decay too slowly,
random information is %rather 
useless.
A typical intersection ellipsoid
is almost as large as the largest.
There is also an intermediate case
where random information is
worse than optimal information,
but only slightly.

\begin{rem}
 The case $\sigma_n \asymp n^{-\alpha} \ln^{-\beta}(n+1)$ with $\alpha>1/2$
 can be extended to $\sigma_n \asymp n^{-\alpha} \phi(n)$ 
 for any slowly varying function $\phi$. 
 \dk{Also in this case, we get $\IE[ \mathcal R_n(\sigma)] \asymp \sigma_{n+1}$.}
% In this case, random information is up to a constant 
% as powerful as optimal information, 
% i.e., $\IE[ \mathcal R_{n,m}(\sigma)] \asymp \sigma_{n+1}$.
\end{rem}

\medskip

\dk{We also} discuss sequences of 
exponential decay. %\dk{and obtain the following}.
We have seen that $\IE[\mathcal R_n(\sigma)]\asymp \sigma_{n+1}$
holds for sequences with
sufficiently fast polynomial decay.
It remains open whether the same holds
for sequences of exponential decay.
\dk{With Theorem~\ref{thm:secondUB} we obtain the following.}

\begin{cor}
\label{cor:exponential}
%Let $m\in\IN\cup \{\infty\}$. 
Assume that $\sigma_n \asymp a^n $
 for some $a\in(0,1)$. Then
\[
  a^n
  \,\preccurlyeq\,
  \IE[ \dk{\mathcal R_n(\sigma)}]
  \,\preccurlyeq\,
  n^2\, a^n.
%  \quad\text{for}\quad
% n<m.
 \]
\end{cor}

%\begin{rem} 
%\dk{Let us briefly compare with the case of polynomial decay.}
Despite the gap, 
\dk{this result} is stronger 
than the result \dk{for polynomial decay} % of Corollary~\ref{cor:polynomial}
if considered from the complexity point of view.
Corollary~\ref{cor:polynomial} states
that there is a constant $c\in(0,\infty)$ such that $cn$ pieces
of random information are at least as good as
$n$ pieces of optimal information.
Corollary~\ref{cor:exponential} states that
there is a constant $c\in(0,\infty)$ such that $n+c\ln n$ pieces
of random information are at least as good
as $n$ pieces of optimal information.
\section{The Proofs}
\label{sec:proofs}
% % % % % % % % % % % % % % % % % % % % % % % % % % % % % % % % %

\noindent
\dk{In the proofs we will use the following tools:}
%In our proofs we shall use a variety of ideas and tools, among others these include
\renewcommand\labelitemi{\tiny$\bullet$}
\begin{itemize}
\item exponential estimates for sums of chi-squared random variables,
\item Gordon's min-max theorem for Gaussian processes,
\item estimates for the extreme singular values of (structured) Gaussian matrices.
\end{itemize}
Before we enter the proofs,
we recall and extend some of our notation.
Let $\sigma=(\sigma_j)_{j=1}^\infty$ 
be a non-increasing sequence of non-negative numbers.
We consider the Hilbert space
$$
 \dk{H_\sigma
 =\set{ x \in \ell_2 \colon
 \sum_{j\in\IN} \left(\frac{x_j}{\sigma_j}\right)^2 <\infty
 },}
 %\set{ x \in \ell_2 \colon
 %x_j=0 \text{ if } \sigma_j=0,\,
 %\sum_{j=1}^\infty \frac{x_j^2}{\sigma_j^2} <\infty
 %}
$$
\dk{using the convention that $a/0=\infty$ for $a>0$ and $0/0=0$},
with inner product 
$$
 \scalar{ x}{ y}_\sigma
 = \sum_{\dk{j\in\IN}} \frac{x_j y_j}{\sigma_j^2}.
$$
%Note that we write $\sum_{j=1}^\infty$ but only take the sum over all $j\in\IN$ for which $\sigma_j$ is positive.
The unit ball of $\dk{H_\sigma}$ is denoted by $\mathcal E_\sigma$.
The matrix $G_n=(g_{ij})_{1\leq i\leq n,j\in\IN}$ for $n\in\IN$ has 
independent standard Gaussian entries.
We want to study the distribution of the random variable
$$
 \mathcal R_n(\sigma)
 = \sup\set{\norm{x}_2 \colon x \in \mathcal E_\sigma,\, G_n x=0}.
$$
Of course, the equation $G_n x=0$
requires that the series
$\sum_{j=1}^\infty g_{ij}x_j$ converges. %s for all $i\leq n$.
For index sets $I\subseteq \IN$ and
$J\subseteq \IN$, we consider the 
 (structured) Gaussian
 $I\times J$-matrices
 $$
  G_{I,J} =\brackets{g_{ij}}_{i\in I, j\in J}
  \quad\text{and}\quad
  \Sigma_{I,J} =\brackets{\sigma_j g_{ij}}_{i\in I, j\in J}.
 $$
 Note that $G_n=G_{[n],\IN}$ and $G_{n,m}=G_{[n],[m]}$, 
 where $[n]$ denotes the set of integers from 1 to~$n$.
 We consider
 $$
 H_J(\sigma)=\set{ x \in \dk{H_\sigma} \colon  x_j=0 \text{ for all } j\in \IN\setminus J }
 $$
 as a closed subspace of the Hilbert space $\dk{H_\sigma}$
 and denote its unit ball by $\mathcal E_{\sigma}^J$.
 The projection of $ x\in \dk{H_\sigma}$ onto $H_J(\sigma)$
 is denoted by $ x_J$.

A crucial role in our proofs is played by estimates 
for the extreme singular values of random matrices.
%This topic has attracted considerable attention over the past years.
We recall some basic facts about singular values.
Let $A$ be a real $r\times k$-matrix,
where we allow that $r=\infty$ or $k=\infty$
provided that $A$ describes a compact operator from \dk{$\IR^k$}
to \dk{$\IR^r$ (with Euclidean norm)}.
For every $j\leq k$,
the $j$th singular value $s_j(A)$ of this matrix 
can be defined as the square-root of the $j$th largest
eigenvalue of the symmetric matrix $A^\top A$,
which describes a positive operator on $\IR^k$.
Note that $s_j(A)=s_j(A^\top)$ if we have $j\leq \min\{r,k\}$.
Our interest lies in the extreme singular values of $A$.
The largest singular value of $A$ is given by
$$
 s_1(A)=
 \sup_{ x \in \IR^k\setminus\{ 0\}}
 \frac{\Vert A x\Vert_2}{\Vert x\Vert_2}
 =\norm{A\colon\IR^k \to \IR^r}.
$$
This number is also called the spectral norm of $A$.
The smallest singular value is given by
$$
  s_k(A)=
  \inf_{ x \in \IR^k\setminus\{ 0\}} 
  \frac{\Vert A x\Vert_2}{\Vert x\Vert_2}.
$$
Clearly, we have $s_k(A)=0$ whenever $k>r$.
If $r\leq k$, it also makes sense to talk about 
the $r$th singular value of $A$.
This number equals the radius of the largest
Euclidean ball that is contained in the image of the unit ball 
of $\IR^k$ under $A$, that is
$$ 
 s_r(A) = \sup \set{ \varrho \geq 0 \colon 
 \varrho \mathbb{B}_2^r \subseteq A ( \mathbb{B}_2^k)},
$$
% egal ob voller Rang oder nicht
where $\mathbb{B}_2^k$ denotes the unit ball in $k$-dimensional
Euclidean space.
These extreme singular values are also defined for noncompact
operators $A$, where $A$ is restricted
to its domain if necessary.
We now turn to the proofs of our results.

% % % % % % % % % % % % % % % %
\subsection{The Upper Bound}
% % % % % % % % % % % % % % % %

We start with an \dk{almost sure} upper bound for 
\dk{the worst case error of the least squares algorithm $A_n$ from \eqref{eq:An}.
The upper bound is given}
%$\mathcal R_n(\sigma)$
in terms of the extreme singular values
of the corresponding (structured) Gaussian matrices. The spectral statistics of random matrices, in particular the behavior of the least and largest singular value, attracted considerable attention over the years and we refer the reader to, e.g., \cite{AGLPTJ2008, BH16, DS01, LPRTJ2005, R2008, RV2008, TV2009, VH2017} and the references cited therein.

\begin{prop}
 \label{prop:radius vs singular values}
 Let $\sigma\in\ell_2$ be non-increasing
 and let $k\leq n$.
 If $G_{n,k}\in\IR^{n\times k}$ has full rank, then
 $$
 %\mathcal R_n(\sigma) 
 \dk{e(A_n)}
 \,\le\, 
 \sigma_{k+1} + 
 \frac{s_1\brackets{\Sigma_{[n],\IN\setminus[k]}}}{s_k\brackets{G_{n,k}}}.
 $$
\end{prop}

\begin{proof}
We first note that $s_k(G_{n,k})>0$
if $G_{n,k}$ has full rank.
\dk{Let $x\in H_\sigma$ with $\Vert x\Vert_\sigma\leq 1$.
We recall that $A_n(x_{[k]})=x_{[k]}$. This yields
\[\begin{split}
 \norm{x-A_n x}_{2} \,&\le\, \norm{x-x_{[k]}}_{2} + \norm{A_n(x-x_{[k]})}_{2}
 \,\le\, \sigma_{k+1} + \norm{G_{n,k}^+ G_n(x-x_{[k]})}_{2} \\
 &\le\, \sigma_{k+1} +\norm{G_{n,k}^+\colon \IR^n \to \IR^k} \cdot \norm{G_n: H_{\IN\setminus [k]}(\sigma) \to \IR^n}.
\end{split}
\]
The norm of $G_{n,k}^+$ is the inverse of the $k$th largest (and therefore the smallest) singular value of the matrix $G_{n,k}$.
The norm of $G_n$ is the largest singular value of the matrix 
\[
 \Sigma =\big(\sigma_j g_{ij}  \big)_{1\leq i \leq n, j> k} \in \IR^{n\times \infty}.
\]
To see this, note that $G_n=\Sigma \Delta$ on $H_{\IN\setminus [k]}(\sigma)$, where the mapping $\Delta\colon H_{\IN\setminus [k]}(\sigma)\mapsto\ell_2$ with $\Delta y=(y_j/\sigma_j)_{j> k}$
is an isomorphism.
This yields the stated inequality.}
\end{proof}

The task now is to bound the $k$-th
singular value of the Gaussian matrix $G_{n,k}$
from below
and the largest singular value of the
structured Gaussian matrix $\Sigma_{[n],\IN\setminus[k]}$
from above.
We start with
the largest singular value 
of the latter.
Let us remark that the question for the order of the expected value of the 
largest singular value 
of a structured Gaussian random matrix has recently been settled by 
Lata\l a, Van Handel, and Youssef \cite{LVHY2017} (see also \cite{BH16, E1988, GHLP2017, L2005, VH2017} for earlier work in this direction). The result we shall use here is due to Bandeira and Van Handel~\cite{BH16}.

\begin{lemma}
\label{lem:BH}
Let $\sigma\in\ell_2$ be non-increasing.
For every $c\in[1,\infty)$
and $n,k\in\IN$, 
we have
\[
\IP\left[
s_{1}\brackets{\Sigma_{[n],\IN\setminus[k]}} 
\geq
 \frac{3}{2}\sqrt{\sum\nolimits_{j>k} \sigma_j^2} + 11c\,\sigma_{k+1} \sqrt{n}
\right]
\leq e^{-c^2 n}.
\]
\end{lemma}

\begin{proof}
Without loss of generality, we may 
assume that $\sigma_{k+1}\neq 0$.
Let us first consider the finite matrix 
$$
 A_m=\Sigma_{[n],[m+k]\setminus[k]}\in\IR^{n\times m}
\quad\text{for}\quad
m\in\IN.
$$
and set
$$
 C_m=\frac{3}{2}\Big(\sum_{j=k+1}^{k+m} \sigma_j^2\Big)^{1/2} 
 + \frac{103\,c}{10}\,\sigma_{k+1} \sqrt{n},
$$
where $A$ and $C$ denote their infinite dimensional variants.
It is proven in \cite[Corollary 3.11]{BH16} that, 
for every $t\geq 0$ (and $\varepsilon=1/2$),
we have
$$
 \IP\bigg[s_{1}(A_m) \,\ge\, 
 \frac{3}{2}\Bigl(\Big(\sum_{j=k+1}^{k+m} \sigma_j^2\Big)^{1/2}+\sigma_{k+1}\sqrt{n}
 +\frac{5\sqrt{\ln(n)}}{\sqrt{\ln(3/2)}} \sigma_{k+1} \Bigr)+t\bigg]
 \le e^{-t^2/2\sigma_{k+1}^2}.
$$
By setting $t=\sqrt{2}c\sigma_{k+1}\sqrt{n}$,
it follows that
$$
 \IP[s_1(A_m)\geq C_m] \leq e^{-c^2n}.
$$
Turning to the infinite dimensional case,
we note that we have $s_1(A)>C$ if and only if
there is some $m\in\IN$ such that $s_1(A_m)>C$.
This yields
$$
 \IP[s_1(A)> C]
 = \IP\left[\exists m\in\IN\colon s_1(A_m)> C\right]
 = \lim_{m\to\infty} \IP[s_1(A_m)> C]
 \leq e^{-c^2n}
$$
since $s_1(A_m)$ is increasing in $m$ and $C\geq C_m$.
\end{proof}

Together with Proposition~\ref{prop:radius vs singular values}
this yields that the estimate
\begin{equation}
\label{eq:intermediate upper bound}
 \dk{e(A_n)}
 \leq 
 \sigma_{k+1} + 
 \frac{\frac{3}{2}\sqrt{\sum\nolimits_{j>k} \sigma_j^2} + 11c\,\sigma_{k+1} \sqrt{n}}{
 s_k\brackets{G_{n,k}}}
\end{equation}
holds with probability at least $1-e^{-c^2n}$ for all $k\leq n$
and $c\geq 1$.
It remains to bound the $k$-th singular value 
of the Gaussian matrix $G_{n,k}$ from below.
It is known from \cite[Theorem~1.1]{RV09} that this number
typically is of order $\sqrt{n}-\sqrt{k-1}$
for all $n\in\IN$ and $k\leq n$.
To exploit our upper bound to full extend, 
the number $k\leq n$ may be chosen such
that the right-hand side of \eqref{eq:intermediate upper bound}
becomes minimal.
We realize that the term $1/s_k(G_{n,k})$
increases with $k$, whereas all remaining terms
decrease with $k$.
However, the inverse singular number achieves its
minimal order $n^{-1/2}$ already for
$k= cn$ with some $c\in(0,1)$.
If $\sigma$ does not decay extremely fast,
this does not lead to a loss 
regarding the other terms of \eqref{eq:intermediate upper bound}.
For instance, we may choose $k=\lfloor n/2\rfloor$
and use the following special case of \cite[Theorem~II.13]{DS01}.

\begin{lemma}
\label{lem:smallest SV rectangular matrix}
Let $n\in\IN$ and $k=\lfloor n/2\rfloor$.
Then
\[
\IP\Big[s_k\brackets{G_{n,k}} \le \sqrt{n}/7 \Big]
\le e^{-n/100}. 
\]
\end{lemma}

\begin{proof}
 It is shown in~\cite[Theorem~II.13]{DS01} that,
 for all $k\leq n$ and $t>0$, we have
 $$
  \IP\left[s_k\brackets{G_{n,k}} \le \sqrt{n}\brackets{1-\sqrt{k/n}-t} \right]
  \leq e^{-n t^2/2}.
 $$
 The statement follows by putting $k=\lfloor n/2\rfloor$ and $t^{-1}=\sqrt{50}$.
\end{proof}

If $\sigma$ decays very fast,
$k=\lfloor n/2\rfloor$
might not be the best choice.
The term $\sigma_{k+1}$ 
in estimate~\eqref{eq:intermediate upper bound}
may be much smaller for $k=n$ 
than for $k=\lfloor n/2\rfloor$.
It is better to choose $k=n$.
In this case,
the inverse singular number is of order $\sqrt{n}$.
We state a result of \cite[Theorem~1.2]{Sz91}.

\begin{lemma}
\label{lem:smallest SV square matrix}
 Let $n\in\IN$ and $t\geq 0$. Then
$$
 \IP\Big[s_n\brackets{G_{n,n}} \le \frac{t}{\sqrt{n}}\Big]
 \le t \sqrt{2e}.
$$
\end{lemma}

This leads to the proof of Theorem~\ref{thm:upper bound random section A} \dk{and \ref{thm:secondUB}}
as presented in the introduction.

\begin{proof}[\bf Proof of Theorem~\ref{thm:upper bound random section A} \dk{and \ref{thm:secondUB}}]
 To prove the first statement, let $k=\lfloor n/2\rfloor$.
 We combine Lemma~\ref{lem:smallest SV rectangular matrix} 
 and Lemma~\ref{lem:BH} for $c=1$ with  Proposition~\ref{prop:radius vs singular values}
 and obtain that
 $$
  \dk{e(A_n)} \leq 78\,\sigma_{k+1} +\frac{21}{2\sqrt{n}}
   \bigg(\sum_{j>\lfloor n/2\rfloor}\sigma_j^2\bigg)^{1/2}
 $$
 with probability at least $1-e^{-n}-e^{-n/100}$.
 The statement follows if we take into account that
 $$
  \sigma_{k+1}^2 \leq \frac{4}{n} 
  \sum_{j=\lfloor n/4\rfloor}^{\lfloor n/2\rfloor}\sigma_j^2.
 $$
 
 To prove the second statement, we set $t=c/s$.
 We combine Lemma~\ref{lem:smallest SV square matrix} 
 and Lemma~\ref{lem:BH} with Proposition~\ref{prop:radius vs singular values}
 and obtain that
 $$
  \dk{e(A_n)}  \leq \sigma_{n+1} + \frac{1}{t}
  \brackets{\frac{3\sqrt{n}}{2}\bigg(\sum_{j>n}\sigma_j^2\bigg)^{1/2}
  + 11c\,n\, \sigma_{n+1} }
 $$
 with probability at least $1-e^{-c^2 n}-t\sqrt{2e}$.
 The rough estimates $\sigma_{n+1}^2 \leq \sum_{j>n} \sigma_j^2$
 and $3\sqrt{n}/2\leq 2cn$ and $1\leq sn$ yield the statement.
\end{proof}

% % % % % % % % % % % % % % % % % %
\subsection{The Lower Bound}
% % % % % % % % % % % % % % % % % %

We want to give lower bounds on the radius of information
$$
 \mathcal R_n(\sigma)
 = \sup\big\{\norm{ x}_2 \colon  x \in \mathcal E_\sigma,\, G_n x = 0\big\}
$$
which corresponds to the difficulty of recovering
an unknown element $ x\in \mathcal E_\sigma$ from the information
$G_n x$ in $\ell_2$.
In fact, our lower bounds already hold for the smaller quantity
$$
 \mathcal R_n^{(k)}(\sigma)
 = \sup\big\{\abs{x_k} \colon  x \in \mathcal E_\sigma,\, G_n x = 0\big\}
$$
which corresponds to the difficulty of recovering
just the $k$th coordinate of $ x$. 
Again, we start with \dk{an almost sure} estimate.

\begin{prop}
\label{prop:lower}
Let $\sigma\in\ell_2$ be non-increasing. 
For all $n,k\in\IN$ with $\sigma_k\neq 0$
we have almost surely
\[
\mathcal R_n^{(k)}(\sigma) \,\ge\,
\sigma_k \brackets{1-\frac{\norm{(g_{ik})_{i=1}^n}_2}{
\sigma_k^{-1} s_n\brackets{\Sigma_{[n],\IN\setminus\{k\}}} + \norm{(g_{ik})_{i=1}^n}
}}.
\]
\end{prop}

\begin{proof}
We may assume that the operator $G_n\colon H_{\IN\setminus\{k\}}(\sigma) \to \IR^n$ 
is onto and that $ g=(g_{ik})_{i=1}^n$ is nonzero
since these events occur with probability 1.  
Observe that
$$
 G_n\brackets{\mathcal E_{\sigma}^{\IN\setminus\{k\}}}
 =\Sigma_{[n],\IN\setminus\{k\}}\brackets{\mathbb{B}_2},
$$
where $\mathbb{B}_2$ is the unit ball of $\ell_2$.
In particular, this implies
$$
 s_n :=
 s_n\brackets{\Sigma_{[n],\IN\setminus\{k\}}}
 = \sup \set{ \varrho \geq 0 \colon 
 \varrho \mathbb{B}_2^n 
 \subseteq \Sigma_{[n],\IN\setminus\{k\}}\brackets{\mathbb{B}_2}}
 >0.
$$
Let $ e^{(k)}$ be the $k$-th standard unit vector in $\ell_2$. 
Then we have 
$$
\| e^{(k)}\|_2=1
\quad\text{and}\quad 
\norm{ e^{(k)}}_\sigma=\sigma_k^{-1}.
$$
Since the image of $\mathcal E_\sigma^{\IN\setminus\{k\}}$ under $G_n$
contains a Euclidean ball of radius $s_n$,
we find an element $\bar{ y}$ of $\mathcal E_\sigma^{\IN\setminus\{k\}}$ 
such that 
$$
 G_n \bar{ y} =
 \frac{s_n\cdot G_n  e^{(k)}}{\| G_n  e^{(k)} \|_2}.
$$ 
For $ y= s_n^{-1} \| G_n  e^{(k)} \|_2\cdot \bar{ y}$, 
we obtain $G_n y = G_n  e^{(k)}= g$ and
\[
\norm{ y}_\sigma
=s_n^{-1} \| G_n  e^{(k)} \|_2 \norm{\bar{ y}}_\sigma
\le s_n^{-1} \| g\|_2.
\]
Then the vector $ z:= e^{(k)}- y$ satisfies $G_n  z=0$ 
and $z_k=1$ as well as  
\[
\norm{ z}_\sigma
\le \norm{ e^{(k)}}_\sigma+\norm{ y}_\sigma
\le \sigma_k^{-1} + s_n^{-1} \| g\|_2.
\]
The statement is obtained by
\[
\mathcal R_n^{(k)}(\sigma) \ge \frac{|z_k|}{\norm{ z}_\sigma} 
= \frac{1}{\norm{ z}_\sigma}. 
\]
\end{proof}

It remains to bound the $n$th singular value of $\Sigma_{[n],\IN\setminus\{k\}}$ 
and the norm of the Gaussian vector $(g_{ik})_{i=1}^n$
with high probability.
For both estimates, we use the following concentration result for chi-square random variables going back to Laurent and Massart \cite[Lemma 1]{LM00}. Alternatively, one could use the concentration of Gaussian random vectors in Banach spaces (see, e.g., \cite[Proposition~2.18]{Ledoux2001}).

\begin{lemma}\label{lem:vector}
For $1 \leq j \leq m$, let $u_j$ be independent centered Gaussian variables
with variance $a_j$.
Then, for any $\delta\in(0,1]$, we have
\begin{align*}
 \IP\bigg[\sum_{j=1}^m u_j^2 \,\le\, (1-\delta) \sum_{j=1}^m a_j\bigg]
 &\leq \exp\brackets{-\frac{\delta^2\norm{ a}_1}{4\norm{ a}_\infty}},\\
 \IP\bigg[\sum_{j=1}^m u_j^2 \,\ge\, (1+\delta) \sum_{j=1}^m a_j\bigg]
 &\leq \exp\brackets{-\frac{\delta^2\norm{ a}_1}{16\norm{ a}_\infty}}.
\end{align*}
\end{lemma}

\begin{proof}
The lemma~\cite[Lemma 1]{LM00} states that, for all $t>0$, we have
\begin{align*}
 &\IP\left[\sum_{j=1}^m u_j^2 \leq 
 \norm{ a}_1 - 2\norm{ a}_2 t\right] 
 \leq e^{-t^2}, \\
 &\IP\left[\sum_{j=1}^m u_j^2 \geq 
 \norm{ a}_1 + 2\norm{ a}_2 t + 2\norm{ a}_\infty t^2 \right] 
 \leq e^{-t^2}.
\end{align*}
The formulation of Lemma~\ref{lem:vector} follows
if we put 
\[
t=\frac{\delta \norm{ a}_1}{2 \norm{ a}_2},
\qquad\text{respectively}\qquad
t=\min\set{
\frac{\delta \norm{ a}_1}{4 \norm{ a}_2}, 
\sqrt{\frac{\delta \norm{ a}_1}{4 \norm{ a}_\infty}}
}.
\]
The desired probability estimate then follows by using 
$\| a\|_2^2\le \| a\|_1 \| a\|_\infty$.
\end{proof}

In particular,
the norm of the Gaussian vector $(g_{ik})_{i=1}^n$
concentrates around $\sqrt{n}$.
To bound the $n$th singular value of $\Sigma_{[n],\IN\setminus\{k\}}$
we shall use Gordon's min-max theorem.
Let us state Gordon's theorem \cite[Lemma 3.1]{Go88} 
in a form which can be found in \cite{HOT15}.

\begin{lemma}[Gordon's min-max theorem]
\label{thm:Gordon min-max}
Let $n,m\in\IN$ and let $S_1\subseteq\IR^n$, $S_2\subseteq\IR^m$ be compact sets. 
Assume that $\psi\colon S_1\times S_2\to\IR$ is a continuous mapping. 
Let $G\in\IR^{m\times n}$, $ u\in\IR^m$, and $ v\in\IR^n$ be independent random objects 
with independent standard Gaussian entries. Moreover, define
\begin{eqnarray*}
\Phi_1(G) & := & \min_{ x\in S_1}\max_{ y\in S_2}
 \Big( \langle  y,G x \rangle + \psi( x, y)\Big), \cr
\Phi_2( u, v)& := & \min_{ x\in S_1}\max_{ y\in S_2} 
 \Big(\| x\|_2\langle  u, y \rangle + \| y\|_2\langle  v, x \rangle
 + \psi( x, y) \Big).
\end{eqnarray*}
Then, for all $c\in\IR$, we have
\[
\IP\big[\Phi_1(G)< c\big] \leq 2\, \IP\big[\Phi_2( u, v) \leq c\big].
\]
\end{lemma}

This yields the following lower bound
on the smallest singular value
of structured Gaussian matrices.
Note that this is a generalization
of Lemma~\ref{lem:smallest SV rectangular matrix}.

\begin{lemma}
\label{lem:smin basic}
Let $A\in\IR^{m\times n}$ be a random matrix
whose entries $a_{ij}$ are centered Gaussian
variables with variance $a_i$ for all $i\leq m$ and $j\leq n$.
Then, for all $\delta\in(0,1)$, we have
\begin{equation}
\label{eq:smin basic}
\IP\left[s_n(A) \leq 
 \sqrt{(1-\delta)\norm{ a}_1}
 - \sqrt{(1+\delta) n \norm{ a}_\infty} \right] 
 \leq 4 \exp\left(-\frac{\delta^2}{16} \min\set{n, 
 \frac{\norm{ a}_1}{\norm{ a}_\infty}}\right).
\end{equation}
\end{lemma}

\begin{proof}
Note that the statement is trivial if $m\leq n$.
We may assume that the $a_i$ are positive
since an additional row of zeros does neither change
$s_n(A)$ nor the norms of the vector $ a$.
We have the identity $A=DG$ where $G\in\IR^{m\times n}$ is
a random matrix with independent standard Gaussian entries
and $D\in\IR^{m\times m}$ is the diagonal matrix 
$$
 D=\diag\brackets{\sqrt{a_1},\hdots,\sqrt{a_m}}.
$$
We want to apply Gordon's theorem for the matrix $G$
and $\psi=0$,
where $S_1$ is the sphere in $\IR^n$ and
$S_2$ is the image of the sphere in $\IR^m$ under $D$.
Then we have
\begin{multline*}
 \Phi_1(G) = \min_{ x\in S_1}\max_{ y\in S_2}
    \langle  y,G x \rangle
 = \min_{\norm{ x}_2=1}\max_{\norm{ z}_2=1}
    \langle D z,G x \rangle\\
 = \min_{\norm{ x}_2=1}\max_{\norm{ z}_2=1}
    \langle  z,A x \rangle
 = \min_{\norm{ x}_2=1} \norm{A  x}_2
 = s_n(A).
\end{multline*}
On the other hand, if $ u\in\IR^n$ and $ v\in\IR^m$
are standard Gaussian vectors, 
the choice of $ z= D u / \Vert D u\Vert_2$ yields
\begin{multline*}
 \Phi_2( u, v)=
 \min_{ x\in S_1} \max_{ y\in S_2}
   \Big(\scalar{ u}{ y}+ \norm{ y}_2 \scalar{ v}{ x}\Big)\\
 =\min_{\norm{ x}_2=1}\max_{\norm{ z}_2=1}
   \Big(\scalar{ u}{D z}+ \norm{D z}_2 \scalar{ v}{ x}\Big)\\
 \geq \min_{\norm{ x}_2=1}
   \Big(\norm{D u}_2+ \frac{\norm{D^2 u}_2}{\norm{D u}_2} 
   \scalar{ v}{ x}\Big)\\
 = \norm{D u}_2- \frac{\norm{D^2 u}_2}{\norm{D u}_2} \norm{ v}_2
 \geq \norm{D u}_2- \sqrt{\norm{ a}_\infty} \norm{ v}_2.
\end{multline*}
Theorem~\ref{thm:Gordon min-max} implies for all $c\in\IR$ that
$$
 \IP\Big[s_n(A)< c\Big]
 \leq 2 \IP\Big[ \Phi_2( u, v) \leq c \Big]
 \leq 2 \IP\Big[ \norm{D u}_2- \sqrt{\norm{ a}_\infty} \norm{ v}_2 
      \leq c \Big].
$$
To obtain the statement of our lemma,
we set $c=\sqrt{(1-\delta)\Vert a\Vert_1} - \sqrt{(1+\delta) n \Vert a\Vert_\infty}$.
By Lemma~\ref{lem:vector}, we have 
\[
\IP\Big[\norm{D u}_2 \,\le\, \sqrt{(1-\delta) \norm{ a}_1}\Big]
\,\le\, \exp\brackets{- \frac{\delta^2\norm{ a}_1}{4\norm{ a}_\infty}}
\]
and 
\[
\IP\Big[\norm{ v}_2 \,\ge\, \sqrt{(1+\delta) n}\Big]
\,\le\, \exp\brackets{-\frac{\delta^2 n}{16}}.
\]
Now the statement is obtained from a union bound.
\end{proof}

We need the statement of Lemma~\ref{lem:smin basic}
for matrices with infinitely many rows,
which is obtained from a simple limit argument.

\begin{lemma}
\label{lem:smin basic infinite}
 Formula~\eqref{eq:smin basic} also holds for $m=\infty$
 provided that $ a\in \ell_1$.
\end{lemma}

\begin{proof}
Again, we may assume that $ a$ is strictly positive.
For $m\in\IN$ let $A_m$ be the sub-matrix consisting
of the first $m$ rows of $A$ and let $ a^{(m)}$
be the sub-vector consisting of the first $m$ entries of $ a$.
We use the notation
\begin{align*}
 c_m(\delta)&=\sqrt{(1-\delta)\Vert a^{(m)}\Vert_1} 
 - \sqrt{(1+\delta) n \Vert a^{(m)}\Vert_\infty}, \\
 p_m(\delta)&=4 \exp\left(-\frac{\delta^2}{16} \min\set{n, 
 \frac{\norm{ a^{(m)}}_1}{\norm{ a^{(m)}}_\infty}}\right),
\end{align*}
where $c(\delta)$ and $p(\delta)$ correspond to the case $m=\infty$.
For any $\varepsilon>0$ with $\varepsilon<\delta/2$
we can choose $m\geq n$ such that 
$c(\delta) \leq c_m(\delta-\varepsilon)$
and $p_m(\delta-\varepsilon) \leq p(\delta-2\varepsilon)$.
Note that we have $s_n(A)\geq s_n(A_m)$
and thus
\begin{multline*}
 \IP\left[s_n(A) \leq c(\delta) \right]
 \leq \IP\left[s_n(A_m) \leq c(\delta) \right]
 \leq \IP\left[s_n(A_m) \leq c_m(\delta-\varepsilon)\right]
 \leq p_m(\delta-\varepsilon) 
 \leq p(\delta-2\varepsilon). 
\end{multline*}
Letting $\varepsilon$ tend to zero yields the statement.
\end{proof}

%This yields the following lower bound for
%the $n$th singular value of $\Sigma_{[n],\IN\setminus\{k\}}$.
%
%\begin{lemma}
%\label{lem:smin}
% Let $\sigma\in\ell_2$ be non-increasing and let $n,k\in\IN$
% and $0<\delta<1$. Then
% \begin{multline*}
% \IP\left[s_n\brackets{\Sigma_{[n],\IN\setminus\{k\}}} \leq 
% \sqrt{(1-\delta)\sum\nolimits_{j>k} \sigma_j^2}
% - \sqrt{(1+\delta) n\, \sigma_{k+1}^2} \right] \\
% \leq 4 \exp\left(- (\delta/4)^2 \min\set{n, 
% \sigma_{k+1}^{-2}\sum\nolimits_{j>k} \sigma_j^2}\right).
% \end{multline*}
%\end{lemma}
%
%\begin{proof}
%This is an immediate consequence of Lemma~\ref{lem:smin basic infinite}
%if we take into account that
%$$
% s_n\brackets{\Sigma_{[n],\IN\setminus\{k\}}}
% = s_n\brackets{\Sigma_{[n],\IN\setminus\{k\}}^\top}
% \geq s_n\brackets{\Sigma_{[n],\IN\setminus[k]}^\top}
%$$
%since erasing rows can only shrink the smallest singular value.
%\end{proof}
%

We arrive at our main lower bound. 

\begin{lemma}
\label{thm:lower}
Let $\sigma\in\ell_2$ be non-increasing and let $n,k\in\IN$ be such that $\sigma_k\neq 0$.
Define
\[
C_k := C_k(\sigma) = \sigma_k^{-2}\sum_{j>k} \sigma_j^2\,. 
\]
Then, for all $\delta\in(0,1)$, 
we have
\[
\IP\left[ \mathcal R_n^{(k)}(\sigma) \,\le\, 
\sigma_k \left(1\,-\,\sqrt{\frac{(1+\delta)n}{(1-\delta)C_k}}\,\right)\right]
\,\leq\, 
5 \exp\brackets{-(\delta/4)^2\, \min\set{n, C_k}}.
\]
%%In particular, if we can choose $k$ such that 
%%$C_k\geq 3n\varepsilon^{-2}$ for some $0<\varepsilon<1$, then
%%\[
%%\IP\Big[\mathcal R_n(\sigma) \,\le\, \sigma_k(1-\varepsilon)\Big]
%%\leq 5\exp\brackets{-n/64}.
%%\]
\end{lemma}

\begin{proof}
First note that, in the setting of Proposition~\ref{prop:lower}, 
the matrix $\Sigma_{[n],\IN\setminus[k]}^\top$ and the vector $(g_{ik})_{i=1}^n$ are independent. 
Lemma~\ref{lem:vector} and Lemma~\ref{lem:smin basic infinite} yield
\begin{align*}
 \|(g_{ik})_{i=1}^n\|_2 \,&\le\, \sqrt{1+\delta}\,\sqrt{n}
 \qquad\text{and}\\
 s_n\brackets{\Sigma_{[n],\IN\setminus[k]}^\top}
 \,&\ge\, \sqrt{1-\delta}\, \sigma_k \sqrt{C_k} \,-\, \sqrt{1+\delta}\, \sigma_{k+1} \sqrt{n} 
\end{align*}
with probability at least $1-5 \exp(-(\delta/4)^2\, \min\{n, C_k\})$. 
Note that we have
$$
 s_n\brackets{\Sigma_{[n],\IN\setminus\{k\}}}
 = s_n\brackets{\Sigma_{[n],\IN\setminus\{k\}}^\top}
 \geq s_n\brackets{\Sigma_{[n],\IN\setminus[k]}^\top}
$$
since erasing rows can only shrink the smallest singular value.
In this case, we have
\begin{align*}
 \frac{\norm{(g_{ik})_{i=1}^n}_2}{\sigma_k^{-1} 
    s_n\brackets{\Sigma_{[n],\IN\setminus\{k\}}} + \norm{(g_{ik})_{i=1}^n}}
 & \leq \frac{\sqrt{1+\delta}\sqrt{n}}{ 
    \sqrt{1-\delta} \sqrt{C_k} - (\sigma_{k+1}/\sigma_k)\sqrt{1+\delta} \sqrt{n} 
    + \sqrt{1+\delta}\sqrt{n}} \cr
 & \leq \frac{\sqrt{1+\delta}\sqrt{n}}{ 
    \sqrt{1-\delta} \sqrt{C_k}}.
\end{align*}
Now the statement is obtained from Proposition~\ref{prop:lower}.
\end{proof}

This also proves Theorem~\ref{thm:lower B} as stated in the previous section.

\begin{proof}[\bf Proof of Theorem~\ref{thm:lower B}]
 We simply apply Lemma~\ref{thm:lower} 
 and choose $\delta=1/2$ \dk{to obtain the desired lower bound for $\mathcal R_n^{(k)}(\sigma)$.
 Since the lower bound is independent of $\sigma_1,\hdots,\sigma_{k-1}$,
 we actually get the same lower bound for $\mathcal R_n^{(k)}(\tilde\sigma)$,
 where $\tilde\sigma$ is obtained from $\sigma$ by replacing the first $k-1$ coordinates with $\sigma_k$.
 To see that the lower bound also holds for $\mathcal R_n^{(1)}(\sigma)$
 (as opposed to $\mathcal R_n^{(k)}(\sigma)$),
 we only need to realize that $\mathcal R_n^{(1)}(\sigma) \ge \mathcal R_n^{(1)}(\tilde\sigma)$,
 where $\mathcal R_n^{(1)}(\tilde\sigma)$ clearly has the same distribution as $\mathcal R_n^{(k)}(\tilde\sigma)$.} 
\end{proof}

\subsection{Corollaries}

In order to optimize the lower bound of Theorem~\ref{thm:lower B},
we may choose $k\in\IN$ such that the
right-hand side of our lower bound becomes maximal.
%We first note that the lower bound is trivial for all $k>n$.
%Theorem~\ref{thm:lower} suggests to choose $k$
%minimal such that $C_k\geq 3n\varepsilon^{-2}$.
If the Euclidean norm of $\sigma$ is large, we simply choose $k=1$.
Taking into account that $\mathcal R_n(\sigma)$
is decreasing in $n$, we immediately arrive at the following result.
%We obtain that less than $\varepsilon^2 C_1/3$
%pieces of random information are useless.

\begin{lemma}
 \label{cor:random info useless}
 Let $\sigma\in\ell_2$ be a nonincreasing sequence of nonnegative numbers
 and let 
  $$
   n_0=\left\lfloor \frac{\varepsilon^2}{3 \sigma_1^2} \sum_{j=2}^\infty  \sigma_j^2\right\rfloor,
   \quad \varepsilon\in(0,1).
  $$ 
  Then $\mathcal R_n(\sigma)\geq \sigma_1(1-\varepsilon)$
  for all $n\leq n_0$
  with probability at least $1-5e^{-n_0/64}$.
\end{lemma}

This leads to a proof of Corollary~\ref{cor:l2 not l2}
which states that random information is useful if and only if $\sigma\in\ell_2$.

\begin{proof}[\bf Proof of Corollary~\ref{cor:l2 not l2}]
 We first consider the case that $\sigma\in\ell_2$.
 Since $\mathcal R_n(\sigma)\leq \sigma_1$, Theorem~\ref{thm:upper bound random section A}
 yields
 $$
  \IE[ \mathcal R_n(\sigma)] \leq 2e^{-n/100} \cdot \sigma_1
  + \frac{156}{\sqrt{n}} 
   \bigg(\sum_{j\geq\lfloor n/4\rfloor}\sigma_j^2\bigg)^{1/2}.
 $$
 The statement is now implied by the fact that $\sigma\in\ell_2$.
 
 For the case that $\sigma\not\in\ell_2$,
 let $0<\varepsilon<1$.
 For $m\in\IN$ let $\sigma^{(m)}$ be the sequence
 obtained from $\sigma$ by replacing the $j$th
 element with zero for all $j>m$.
 For any $N\geq n$, we can choose $m\in\IN$
 such that
 $$
  \frac{\varepsilon^2}{3 \sigma_1^2} \sum_{j=2}^m  \sigma_j^2
  \geq N
 $$
 since $\sigma\not\in\ell_2$.
 The first part of this corollary yields that
 \begin{align*}
  \IP\left[\mathcal R_n(\sigma)\geq \sigma_1(1-\varepsilon)\right]
  & \geq \IP\left[\mathcal R_n(\sigma^{(m)})\geq \sigma_1(1-\varepsilon)\right]\cr
  & \geq \IP\left[\mathcal R_N(\sigma^{(m)})\geq \sigma_1(1-\varepsilon)\right]
   \geq 1-5\exp\brackets{-N/64}.
 \end{align*}
 Since this holds for any $N\geq n$, we get that
 the event $\mathcal R_n(\sigma)\geq \sigma_1(1-\varepsilon)$
 happens with probability 1 for any $\varepsilon\in(0,1)$.
 This yields the statement 
 since the event $\mathcal R_n(\sigma)\geq \sigma_1$
 is the intersection of countably many such events.
\end{proof}

We now apply our general estimates for $\mathcal R_n(\sigma)$ to 
specific sequences $\sigma$
to prove the statements of Corollaries~\ref{cor:polynomial} and \ref{cor:exponential}.

\begin{proof}[\bf Proof of Corollary~\ref{cor:polynomial}]
\emph{Part 1.} %We start with the first equivalence of Corollary~\ref{cor:polynomial}.
The upper bound \dk{in the first equivalence} 
is trivial since \dk{$R_n(\sigma)\leq \sigma_1$} almost surely.
\dk{The lower bound follows immediately from Corollary~\ref{cor:l2 not l2}.}
%To prove the lower bound it is enough to consider the case 
%$m\in\IN$ 
%and the finite sequence
%$$
% \sigma'_j=\min\set{1,j^{-\alpha}\brackets{1+\ln j}^{-\beta}},
% \quad\text{for}\quad j\leq m,
%$$
%where $0\le \alpha \leq 1/2$ and $\beta\in\IR$
%with $\beta\leq 1/2$ for $\alpha=1/2$, 
%and $\beta\ge0$ for $\alpha=0$.
%The general case follows from the fact that 
%%$\mathcal R_n(\sigma)$ is homogeneous and monotonic with respect to $\sigma$, i.e., 
%$\sigma\ge C\sigma'$ implies $\mathcal R_n(\sigma)\ge C\mathcal R_n(\sigma')$
%for all $n$. 
%Lemma~\ref{cor:random info useless} for $\varepsilon=1/2$ yields that
%we have $\mathcal R_n(\sigma)\geq 1/2$ for all $n\leq n_0$
%with probability at least $1-5\exp(-n_0/64)$
%if we put
%\[
% n_0=\left\{\begin{array}{cl}
%        \displaystyle \left\lfloor \frac{(m-2)m^{-2\alpha}}{12 (1+\ln m)^{\max\{2\beta,0\}}} \right\rfloor
%        &
%        \text{for} \quad \alpha<1/2,\vspace*{2mm}
%        \\
%       \displaystyle \left\lfloor \frac{(\ln m-1)}{12 (1+\ln m)^{\max\{2\beta,0\}}} \right\rfloor
%        &
%        \text{for} \quad \alpha=1/2,\, \beta<1/2,\vspace*{2mm}
%        \\
%        \displaystyle \left\lfloor \frac{(\ln\ln m-1)}{12} \right\rfloor
%        &
%        \text{for} \quad \alpha=\beta=1/2.
%        \end{array}\right.
%\]
%This yields the statement on the expected value since $\mathcal R_{n,m}(\sigma)\geq 0$ almost surely.

\emph{Part 2.} \dk{To} prove the second equivalence of Corollary~\ref{cor:polynomial},
it is enough to consider %a specific sequence.
\dk{the sequence}
%Given $\beta>1/2$ and $m\in\IN\cup\{\infty\}$, we set
$$
 \sigma_j = j^{-1/2} (1+\ln j)^{-\beta}
 \quad\text{for}\quad
 \dk{j\in\IN} %\leq m.
$$
\dk{with $\beta>1/2$.}
Note that we have for any \dk{$k\in\IN$ that %$1< k< m<\infty$ that
 $$ 
  \sum_{j=k+1}^\infty \sigma_j^2 
% = \sum_{j=k+1}^m j^{-1} (1+\ln j)^{-2\beta}
  \asymp \ln^{1-2\beta}(k), % - \ln^{1-2\beta}(m),
 $$
 where} the implied constants depend only on $\beta$.
 Now it follows from %the first part of 
 Theorem~\ref{thm:upper bound random section A}
 and from Theorem~\ref{thm:lower B}
 for $k = \lceil c_\beta' n/(1+\ln n)\rceil$
 with some $c'_\beta >0$ that
%there are constants $c_\beta, C_\beta>0$ such that
\[
 \mathcal R_n(\sigma) \asymp n^{-1/2} (1+\ln n)^{1/2-\beta}
\]
 with probability at least $1- 7 e^{-n/100}$,
 where the implied constants depend only on $\beta$.
 The statement for the expected value follows from $0\leq R_n(\sigma)\leq 1$.

\emph{Part 3.} %We now prove 
\dk{In} the third equivalence of Corollary~\ref{cor:polynomial},
the lower bound is trivial \dk{and even holds almost surely}. 
%since $R_n(\sigma)\geq \sigma_{n+1}$ almost surely.
To prove the upper bound, it is enough to consider %the case $m=\infty$ and 
the sequence
$$
 \sigma_j = \min\set{1, j^{-\alpha} (1+\ln j)^{-\beta}}
 \quad\text{for}\quad
 j\in\IN,
$$
where $\alpha>1/2$ and $\beta\in\IR$.
Theorem~\ref{thm:upper bound random section A}
yields for large $n$ that
$$ 
	\mathcal R_n(\sigma)^2
	\preccurlyeq \frac{1}{n} \sum_{j\geq\lfloor n/4\rfloor} \sigma_j^2 
  \preccurlyeq n^{-2\alpha}(1+\ln n)^{-2\beta}
$$
with probability at least $1- 2 e^{-n/100}$
and implied constants only depending on $\alpha$ and~$\beta$.
This yields the statement since $\mathcal R_n(\sigma)\leq 1$ almost surely.
\end{proof}

\begin{proof}[\bf Proof of Corollary~\ref{cor:exponential}]
 The lower bound follows from the trivial estimate 
 $\mathcal R_n(\sigma) \ge \sigma_{n+1}$.
 To prove the upper bound, we consider the case %$m=\infty$ and
 $\sigma_j = a^{j-1}$ for all $j \in\IN$.
 The general case follows from the monotonicity and homogenity
 of $\mathcal R_n(\sigma)$ with respect to~$\sigma$.
 We use %the second part
 \dk{Theorem~\ref{thm:secondUB}}.
 We choose $c\in[1,\infty)$ such that $e^{-c^2}\leq a$.
 Note that there is some $b\in(0,\infty)$ such that
 $$
  \bigg(\sum_{j>n}\sigma_j^2\bigg)^{1/2} = \frac{b\,a^n}{14}
 $$
 for all $n\in\IN$.
 Theorem~\ref{thm:secondUB} yields for all $t\geq b n a^n$ that
 $$
  \IP[\mathcal R_n(\sigma)\geq t]
  \leq a^n + \frac{b\,n\,a^n\,c\sqrt{2e}}{t}.
 $$
 This yields that
 $$
  \IE[\mathcal R_n(\sigma)]
  = \int_0^1 \IP[\mathcal R_n(\sigma)\geq t]~\d t
  \leq a^n + b n a^n + n a^n \int_{b n a^n}^1 \frac{b c\sqrt{2e}}{t}~\d t
  \preccurlyeq n^2 a^n,
 $$
 as it was to be proven.
\end{proof}

% % % % % % % % % % % % % % % % % % %
\section{Alternative approaches}
% % % % % % % % % % % % % % % % % % %
\noindent

\dk{In this section we present alternative ways to estimate the radius of random information from above and below.
We choose to do this because these approaches give a better insight into the geometric aspect of the problem.
%than the previous approach via least squares estimators.
The results, however, are slightly weaker than those obtained in Section~\ref{sec:proofs}.
The upper bound is weaker since it is not constructive
and the lower bound is weaker since it requires a little more than $\sigma\not\in\ell_2$.
For these geometric approaches, we restrict to the case of 
finite sequences $\sigma$ with $\sigma_j=0$ for all $j>m$.
We write $\mathcal E_\sigma^m$ when we consider the 
ellipsoid $\mathcal E_\sigma$ as a subset of $\IR^m$.

% % % % % % % % % % % % % % % % % % % % % % % % % % % % % % % % % % % % % % % % % % 
\subsection{Upper bound via the lower $M^\ast$-estimate }\label{sec:M-star bound}
% % % % % % % % % % % % % % % % % % % % % % % % % % % % % % % % % % % % % % % % % %

We present an alternative proof of our main upper bound. 
%Theorem~\ref{thm:upper bound random section A}
As already explained in the introduction, the radius of information can also be expressed as the radius of the ellipsoid that is obtained by slicing the $m$-dimensional ellipsoid $\mathcal E_\sigma^m$ with a random \dk{subspace} of codimension $n$. 
To estimate the radius from above, we use a result of Gordon 
from \cite{Go88} on estimates of the Euclidean norm against a 
norm induced by a symmetric convex body $K$ on large subsets 
of Grassmannians. 
Note that the first result in this direction had been established 
by V.D. Milman in \cite{M1985}}.
%The essential quantity that appears is the $M^\ast$-estimate of $K$, which can be easily handled for the ellipsoid $K=\mathcal E_\sigma^m$.

% % % % % % % % % % % % % % % % % % % % % % % % % % % % % % % % % %
%\subsubsection*{Convex bodies and the $M^\ast$-estimate of Gordon} 
% % % % % % % % % % % % % % % % % % % % % % % % % % % % % % % % % %
%David: Ich finde Subsubsections machen das ganze eher unübersichtlich.
\medskip

\dk{We start with} some notation and background information. Let $K \subseteq \IR^m$ be an origin symmetric convex body, i.e., a compact and convex set with non-empty interior such that $x\in K$ implies $-x\in K$. We define the quantity
\[
 M^\ast (K) = \int_{\IS^{m-1}} h_K(x)\, \mu (\dint x),
\]
where $\IS^{m-1}$ is the unit Euclidean sphere in \dk{$\IR^m$}, integration is with respect to the normalized surface measure $\mu=\mu_{m-1}$ on $\IS^{m-1}$, and $h_K:\IS^{m-1} \to \IR$ is the support function of $K$ given by
\[
 h_K(x) = \sup_{y\in K} \, \langle x , y \rangle.
\]
Obviously, the support function is just the dual norm to the norm $\|\cdot\|_K$ induced by~$K$, i.e., if $K^\circ = \{y\in \IR^m \colon \langle y,x\rangle \leq 1 \,\,\forall x\in K \}$ is the so-called polar body of $K$, then $h_K(x) = \|x\|_{K^\circ}$. Since for $x\in \IS^{m-1}$ the support function quantifies the distance from the origin to the supporting hyperplane orthogonal to $x$, the quantity $M^\ast(K)$ is simply (half) the mean width of the body $K$.

\begin{rem}
In the theory of asymptotic geometric analysis, the quantities $M^\ast(K)$ together with
\[
M(K) := \int_{\IS^{m-1}} \|x\|_K \,\mu(\dint x)
\]
play an important r\^ole since the work of V.D. Milman on a quantitative version of Dvoretzky's theorem on almost Euclidean subspaces of a Banach space. Using Jensen's inequality together with polar integration and Urysohn's inequality, it is not hard to see  that
\[
M(K)^{-1} \leq \textrm{vrad}(K) \leq M^\ast(K) = M(K^\circ),
\]
where $\textrm{vrad}(K):= (|K|/|B_2^m|)^{1/m}$ is the volume radius of $K$ (here $|\cdot|$ stands for the $m$-dimensional Lebesgue measure). For isotropic convex bodies in $\IR^m$ (i.e., convex bodies of volume $1$ with centroid at the origin satisfying the isotropic condition -- we refer to \cite{AAM15} for details), this immediately yields
\[
M^\ast(K) \geq \textrm{vrad}(K) \geq c\sqrt{m},
\]
for some absolute constant $c\in(0,\infty)$. The question about upper bounds for $M^\ast(K)$ with $K$ in isotropic position has been essentially settled by E. Milman in \cite[Theorem 1.1]{EM2015} who proved that
\[
M^\ast(K) \leq C L_K \sqrt{m} \log^2m,
\]
with absolute constant $C\in(0,\infty)$. In fact, the $\sqrt{m}$-term is optimal and also the logarithmic part (up to the power). The optimality of the $L_K$-term is intimately related to the famous hyperplane conjecture. For a detailed exposition, we refer the reader to \cite{AAM15} and the references cited therein.
\end{rem}

We continue with the so-called lower $M^*$-estimate.
\dk{The first estimate of this type was proved by V.D.~Milman in \cite{M1985}, see also \cite{M1985_II}. 
We use the asymptotically optimal version obtained by 
Pajor and Tomczak-Jaegermann in \cite{PTJ1986}
with improved constants from \cite{Go88}. 
For the precise formulation, we refer to \cite[Theorem 7.3.5]{AAM15}.}
%We continue with the so-called lower $M^*$-estimate\textcolor{red}{, which we use in the form obtained by \cite{Go88}}. For the formulation used here see \cite[Theorem 7.3.5]{AAM15}. Note that this is an improvement in the sense of constants for an asymptotically optimal result that had been obtained before by Pajor and Tomczak-Jaegermann in \cite{PTJ1986}. The first estimate of this type was proved by V.~Milman \textcolor{red}{in \cite{M1985} (see also \cite{M1985_II})}.

\begin{prop}\label{thm:gordon}
  For $n \in \IN$, define 
  \[ a_n = \frac{\sqrt{2} \Gamma \left( \frac{n+1}{2} \right)}{\Gamma \left( \frac{n}{2} \right)} 
  = \sqrt{n} \left(1-\frac{1}{4n} + \mathcal O(n^{-2}) \right).\]
  Let $K$ be the unit ball of a norm $\| \cdot \|_K$ on $\IR^m$. For any $\gamma\in (0,1)$ and $1\le n < m$ there exists a subset $\mathcal{B}$ in the Grassmannian $\mathbb G_{m,m-n}$ of $n$-codimensional linear subspaces of $\IR^m$ with Haar measure at least
  \[ 1 - \frac{7}{2} \exp \left( - \frac{1}{18} (1-\gamma)^2 a_n^2 \right)   \]
such that for any $E_n \in \mathcal{B}$ and all $x\in E_n$ we have
\[ \frac{\gamma a_n}{a_m M^\ast (K)} \|x\|_{2} \le \|x\|_K.\] 
\end{prop}

We should observe here that the distribution of the kernels of the Gaussian matrices $G_n$ is the uniform distribution, i.e., the distribution of the Haar measure, on the Grassmann manifold $\mathbb G_{m,m-n}$. This follows immediately from the rotational invariance of both measures on $\mathbb G_{m,m-n}$. Hence, the probability estimate in 
%Gordon's 
the theorem is exactly with respect to the probability on the kernels we use elsewhere.

% % % % % % % % % % % % % % % % % % % % % % % % % % % % % % % % % % % % % % % %
%\subsubsection*{Bounding the radius of information via Gordon's $M^\ast$-estimate}
% % % % % % % % % % % % % % % % % % % % % % % % % % % % % % % % % % % % % % % %
\medskip

We want to apply 
%Gordon's 
the lower $M^*$-estimate 
%for an alternative proof of our main upper bound on the radius of information (in the finite dimensional setting).
%Proposition~\ref{thm:gordon} immediately yields an upper bound on the radius $\rad(G_{n,m},K)$ for every convex and symmetric body $K$.
to obtain an upper bound on the radius of information.
For this note that the ellipsoid $\mathcal E_\sigma^m$ satisfies 
\[
(\mathcal E_\sigma^m)^\circ = \left\{x\in\IR^m\colon \sum_{j=1}^m \sigma_j^2x_i^2 \leq 1 \right\}.
\]
This implies $h_{\mathcal E_\sigma^m}^2(x) = \sum_{j=1}^m \sigma_j^2 x_j^2$, 
and therefore,
\[
 M^\ast (\mathcal E_\sigma^m)^2  
	%= \int_{\IS^{m-1}} h_{\mathcal E_\sigma^m}(x) \,\mu(\dint x) 
  \le \int_{\IS^{m-1}} h_{\mathcal E_\sigma^m}(x)^2 \,\mu(\dint x) 
 =  \sum_{j=1}^m \sigma_j^2 \int_{\IS^{m-1}} x_j^2 \,\mu(\dint x)  
  = \frac1m \sum_{j=1}^m \sigma_j^2.
\]
A direct application of Proposition~\ref{thm:gordon} with $K=\mathcal E_\sigma^m$ 
leads to the upper bound
\[
    %\rad(G_{n,m},\mathcal E_\sigma^m)^2 
    \mathcal R_n(\sigma) \le \frac{C}{n} \sum_{j=1}^m \sigma_j^2
\]
with an absolute constant $C\in(0,\infty)$.
This estimate is not very good if the semi-axes $\sigma_j$ decay quickly.
A better estimate can be obtained
by switching from $\mathcal E_\sigma^m$ to its intersection 
with a Euclidean ball of small radius, 
a renorming argument going back to Pajor and 
Tomzcak-Jaegermann~\cite{PTJ1985}.

\begin{prop}\label{thm:gordon_applied}
  There exists an absolute constant $C\in(0,\infty)$ such that for any 
	non-increasing finite sequence $\sigma\in\ell_2$ 
  and all $n\in\IN$, we have   
  \[
  %\rad(G_{n,m},\mathcal E_\sigma^m)^2 
  \mathcal R_n(\sigma) \le \frac{C}{ n} \sum_{j>n/C} \sigma_j^2
  \]
  with probability at least $1-\frac{7}{2} \, \exp(-n/32)$.
\end{prop}

% Mit beliebig kleiner Wahrscheinlichkeit (also mit gamma):
%\begin{prop}\label{thm:gordon_applied}
%  There exists an absolute constant $C>0$ such that for any non-increasing sequence $\sigma\in\ell_2$, all $m,n\in\IN$ and all $\gamma \in (0,1)$, we have   
%  \[
%  \rad(G_{n,m},\mathcal E_\sigma^m)^2 \le \frac{C}{\gamma^2 n} \sum_{j>\gamma^2 n/C} \sigma_j^2
%  \]
%  with probability at least $1-\frac{7}{2} \, \exp(-(1-\gamma)^2 n/18)$.
%%  In particular, there exists a constant $C\in(0,\infty)$ such that for all $m,n\in\IN$ with $n < m$ and all  $\sigma\in \IR^m$
%%  \[
%%    \IE \big[\rad(G_{n,m},\mathcal E_\sigma^m)\big] \le C \frac{\|\sigma\|_{2}}{\sqrt{n}}.
%%  \]
%\end{prop}

\begin{proof}
Let $K_\varrho$ be the intersection of the ellipsoid $\mathcal E_\sigma^m$ with a centered Euclidean ball of radius~$\varrho>0$.
For all $x\in\IR^m$, $y\in K_\varrho$ and $k\le m$, Cauchy-Schwarz inequality yields
\[
 \langle x,y \rangle^2 \le 2 \left( \sum_{j=1}^k x_j y_j \right)^2 + 2 \left( \sum_{j=k+1}^m x_j y_j \right)^2
 \le 2 \varrho^2 \sum_{j=1}^k x_j^2 + 2 \sum_{j=k+1}^m \sigma_j^2 x_j^2
\]
and thus the same upper bound holds for $h_{K_\varrho}(x)^2$.
We obtain that
\[
 M^\ast (K_\varrho)^2 = \left(\int_{\IS^{m-1}} h_{K_\varrho}(x) \,\mu(\dint x) \right)^2
  \le \int_{\IS^{m-1}} h_{K_\varrho}(x)^2 \,\mu(\dint x)
  \le \frac{2}{m} \left(k \varrho^2 + \sum_{j=k+1}^m \sigma_j^2 \right).
\]
Proposition~\ref{thm:gordon} tells us that for any $\gamma \in (0,1)$ there exists a subset $\mathcal B$ of $\mathbb G_{m,m-n}$ with measure at least $1 - \frac{7}{2} \, \exp(-(1-\gamma)^2 n/18)$ such that 
\[
 \rad(K_\varrho \cap E_n)^2 \le \frac{a_m^2 M^\ast (K_\varrho)^2}{\gamma^2 a_n^2} 
 \le \frac{c}{\gamma^2 n} \left(k \varrho^2 + \sum_{j=k+1}^m \sigma_j^2 \right)
% \frac{\gamma a_n}{a_m M^\ast (\mathcal E_\sigma^m)} \|x\|_{2} \le \|x\|_F \quad \text{ for any } E_n\in \mathcal B \text{ and } x \in E_n,
\]
for any $E_n\in \mathcal B$ and an absolute constant $c\in(0,\infty)$.
Choosing $\varrho$ such that $k\varrho^2=\sum_{j=k+1}^m \sigma_j^2$ and $k=\lfloor \gamma^2 n/4c\rfloor$ yields
\[
 \rad(K_\varrho \cap E_n)^2 \le \frac{\varrho^2}{2}.
\]
This clearly implies that also
\[
 \rad(\mathcal E_\sigma^m \cap E_n)^2 \le \frac{\varrho^2}{2}.
\]
%which is exactly the stated inequality.\\
For simplicity, we choose $(1-\gamma)^2=1/2$ and obtain the stated inequality.
%
%The derivation of the estimate for the expectation from this exponential probability estimate is standard. Choose a special $\gamma\in (0,1)$, say with $(1-\gamma)^2/18 = 1/20$, and use that $r(G_n)$ is always at most $\sigma_1$ and $\|\sigma\|_{2}$ is at least $\sigma_1$ to conclude that 
%\[
% \IE \big[\rad(G_{n,m},\mathcal E_\sigma^m)\big] \le C\,  \frac{\|\sigma\|_{2}}{ \gamma \sqrt{n}} + \frac{7}{2} \, \e^{- n/20} \le C'\, \frac{\|\sigma\|_{2}}{\sqrt{n}}
%\] 
%for a suitable constant $C'\in(0,\infty)$ and all $n,m\in\IN$ with $n\le m$.
\end{proof}

%We note that Proposition~\ref{thm:gordon_applied} also yields the existence of an absolute constant $C'>0$ such that
%\[
%    \IE \big[\rad(G_{n,m},\mathcal E_\sigma^m)\big] \le \frac{C'}{\sqrt{n}} \sum_{j>n/C'} \sigma_j^2.
%\]}

\subsection{Elementary lower bound}

In this section we prove the following lower bound. 

\begin{prop} \label{prop:altlow}
\dk{For any $\varepsilon\in(0,1)$ and $c>0$ there is a constant $c'>0$ such that the following holds.
If $\sigma_m \geq c m^{-\alpha}$ for some $m\in\IN$ and
$\alpha \in(0,1/2)$ then %$\rad(N_{n,m},\mathcal{E}_\sigma^m) \geq c'$ 
$\mathcal R_n(\sigma)\ge c'$ holds for all 
$n < m^{1-2\alpha}$ with probability at least $1-\varepsilon$.}
\end{prop}

To obtain this lower bound, 
we first consider the problem of just recovering the first coordinate $x_1$ of $x$ in 
the unit ball $\mathbb{B}_2^m$ of $\dk{\IR^m}$. 
The corresponding radius of information is given by
$$
 \widetilde{\rad}(G_n,\mathbb{B}_2^m)^2 =
 \sup \set{ x_1 \colon \norm{x}_{2}=1 \text{ with } \dk{G_n} x=0 }.
$$

\begin{lemma}\label{lem:carl2}
  For $n<m$, we have
  \begin{equation} \label{eq:intprob1}
   \IE\, \big[\widetilde{\rad}(G_n,\mathbb{B}_2^m)^2\big] = \frac{m-n}{m}. 
 \end{equation}
 In particular, for any $\varepsilon\in \left(0,1\right)$, we have
 \begin{equation} \label{eq:intprob2}
  \widetilde{\rad}(G_n,\mathbb{B}_2^m)^2 \geq 1 - \frac{n}{\varepsilon m} 
 \end{equation}
 with probability at least $1-\varepsilon$.
\end{lemma}

\begin{proof}
Let $k=m-n$.
To prove \eqref{eq:intprob1}, we observe that we want to compute the expectation of the
random variable
$$
 \widetilde{\rad}(\dk{G_n},\mathbb{B}_2^m)^2 =
 \max \big\{ \scalar{x}{y} \colon x\in E,\, \norm{x}_{2}=1\big\},
$$
where $E$ is uniformly distributed on $\mathbb{G}_{m,k}$ and $y=e^{(1)}$ is fixed.
Involving an orthogonal transformation of the coordinate system,
we may also fix the subspace 
$$
 E=\left\langle e^{(1)},\hdots,e^{(k)}\right\rangle
$$
and assume that $y$ is uniformly distributed on the sphere.
This does not change the distribution of $\widetilde{\rad}(\dk{G_n},\mathbb{B}_2^m)^2$.
By the Cauchy-Schwarz inequality, the maximum is attained for 
$$
 x=\frac{P_E(y)}{\norm{P_E(y)}_2},
$$
where $P_E$ denotes the orthogonal projection on $E$.
We obtain
$$
 \widetilde{\rad}(\dk{G_n},\mathbb{B}_2^m)^2 = \norm{P_E(y)}_{2}^2
 = \sum_{j=1}^k y_j^2.
$$
We observe that $\IE(y_j^2)=1/m$ for all $j\leq m$,
since these terms are equal and sum up to 1.
% maximal squared first coordinate of a unit vector $x$ in a random $k=m-n$-dimensional 
% subspace $E$ of $\IR^m$. Instead of taking a random subspace $E$ and a fixed vector $x$, 
% we can also take a random unit vector $x$ and compute the expectation of the squared 
% norm of the orthogonal projection $P_Ex$ onto a fixed $k$-dimensional subspace $E$. 
% Taking for $E$ the span of the first $k$ standard unit vectors $e_1,\dots,e^{(k)}$, we obtain
% \[
%  \IE\, [\tilde{r}(N_n)^2] = \IE \left[ \|P_Ex\|_{\IR^m}^2 \right]= \IE \left[ \sum_{j=1}^k x_j^2 \right] 
% = \frac{k}{m} \, \IE \left[\sum_{j=1}^m x_j^2\right] = \frac{k}{m} \, \IE \left[\|x\|_{\IR^m}^2\right] = \frac{k}{m}.
% \]
This shows \eqref{eq:intprob1}. Estimate \eqref{eq:intprob2} is a direct 
consequence of \eqref{eq:intprob1} taking into account that $0\leq \widetilde{\rad}(\dk{G_n},\mathbb{B}_2^m) \leq 1$.
\end{proof}

\begin{proof}[Proof of Proposition \ref{prop:altlow}]
If we choose $\dk{G_n}$ satisfying \eqref{eq:intprob2}, by definition of $\widetilde{\rad}(\dk{G_n},\mathbb{B}_2^m)$ and 
compactness of the unit sphere of $\IR^m$, we find $x\in \IR^m$ with $\|x\|_{2}=1$ and $\dk{G_n} x=0$ satisfying 
\[
 x_1^2 \geq 1 - \frac{n}{\varepsilon m} . 
\]
Thus $x_1$ is already rather close to $1$ which implies that the other coordinates can not be too big. Indeed we find
\[
\|x\|_\sigma^2 \leq x_1^2 + \frac{1}{\sigma_m^2} (1-x_1^2) \leq 1 + \frac{1}{\sigma_m^2} 
\frac{n}{\varepsilon m} \leq 1 + \frac{1}{\varepsilon m^{2 \alpha}\sigma_m^2}
  \leq  1 + \frac{1}{\varepsilon c^2}.
\]
Rescaling $x$ yields 
\[
 \rad(\dk{G_n},\mathcal{E}_\sigma^m)^2 \ge \frac{\varepsilon c^2}{1 + \varepsilon c^2}
\]
and finishes the proof.
\end{proof}

\begin{ack}
%\mario{We thank the anonymous referee whose comments lead to a substantial improvement of Section~4.}
\dk{We thank several colleagues for valuable remarks and comments.}
Part of the work was done during a special semester at the 
Erwin Schr\"odinger International Institute for Mathematics and Physics (ESI) in Vienna. 
We thank ESI 
for the hospitality.
A. Hinrichs and D. Krieg are supported by the Austrian Science Fund (FWF) 
Project F5513-N26, which is part of the Special Research 
Program ``Quasi-Monte Carlo Methods: Theory and Applications''. 
J. Prochno is supported by the Austrian Science Fund (FWF) Project P32405 ``Asymptotic Geometric Analysis and Applications'' as well as a visiting professorship from Ruhr University Bochum and its Research School PLUS.
\end{ack}

\bibliographystyle{plain}
\bibliography{random_info}

\end{document}